\documentclass[notitlepage,11pt,reqno]{amsart}
\usepackage{amssymb,amsmath,amscd,amsthm,xspace,color,tocvsec2}
\usepackage[mathscr]{euscript}
\usepackage{enumerate}
\usepackage[breaklinks=true]{hyperref}
\usepackage[letterpaper]{geometry}
\usepackage[all, cmtip]{xy}
\usepackage[ampersand]{easylist}

\usepackage{url}

\usepackage{breakurl}

\newcommand{\gkm}{\hat{\fg}_{- \kappa + 2\kappa_c}}
\newcommand{\lt}{L_{-\operatorname{Tate}}}
\newcommand{\csi}{C^{\frac{\infty}{2} + *}}
\newcommand{\gkmod}{\gk\operatorname{-mod}}
\newcommand{\dcs}{ \wfl \backslash W_\lambda / \wstab}
\newcommand{\wfl}{W_{\operatorname{f}, \lambda}}

\newcommand{\wstab}{W^\circ_\lambda}
\newcommand{\wstabf}{W^\circ_{\operatorname{f}, \lambda}}

\newcommand{\Wf}{W_{\operatorname{f}}}
\newcommand{\ch}{\operatorname{ch}}

\newcommand{\Hom}{\operatorname{Hom}}

\newcommand{\gr}{\operatorname{gr }}

\newcommand{\g}{\hat{\mathfrak{g}}}
\newcommand{\fn}{\mathfrak{n}}
\newcommand{\h}{\mathfrak{h}}

\newcommand{\halpha}{\check{\alpha}}

\newcommand{\fsl}{\mathfrak{sl}}

\newcommand{\id}{\text{id}}

\newcommand{\C}{\ensuremath{\mathbb{C}}}

\newcommand{\Q}{\ensuremath{\mathbb{Q}}}

\newcommand{\Z}{\ensuremath{\mathbb{Z}}}

\newcommand{\comment}[1]{}

\newcommand{\OO}{\mathscr{O}}

\newcommand{\Ext}{\operatorname{Ext}}

\newcommand{\Phf}{{}^{\operatorname{f}}\Phi^\vee}

\newcommand{\gk}{\g_\kappa}

\newtheorem{theo}[equation]{Theorem}

\newtheorem{lemma}[equation]{Lemma}
\newtheorem{cor}[equation]{Corollary}
\newtheorem{pro}[equation]{Proposition}

\theoremstyle{definition}
\newtheorem{ex}[equation]{Example}
\newtheorem{defn}[equation]{Definition}
\newtheorem{re}[equation]{Remark}

\newcommand{\fg}{\mathfrak{g}}
\newcommand{\cw}{\mathscr{W}}
\numberwithin{equation}{section}
\newcommand{\scc}{\mathscr{C}}
\newcommand{\cW}{\mathscr{W}}
\begin{document}

\dedicatory{In memory of N.B.}

\title{Semi-infinite cohomology and the linkage principle for $ \mathscr{W}$-algebras}
\author{Gurbir Dhillon}
\date{Spring 2019}
\begin{abstract} Let $\mathfrak{g}$ be a simple Lie algebra, and let $\cW_\kappa$ be the affine $\cW$-algebra associated to a principal nilpotent element of $\mathfrak{g}$ and level $\kappa$. We explain a duality between the categories of smooth $\mathscr{W}$ modules at levels $\kappa + \kappa_c$ and $-\kappa +  \kappa_c$, where $\kappa_c$ is the critical level. Their pairing amounts to a construction of semi-infinite cohomology for the $\cW$-algebra. 

As an application, we determine all homomorphisms between the Verma modules for $\cW$, verifying a conjecture from the conformal field theory literature of de Vos--van Driel. Along the way, we determine the linkage principle for Category $\mathscr{O}$ of the $\cW$-algebra.  \end{abstract}

\maketitle

\section{Introduction}

In a celebrated work, Feigin and Fuchs computed the space of intertwining operators between Verma modules for the Virasoro algebra \cite{test}. As a striking consequence of their calculation, the full subcategories of Verma modules at central charges $c$ and 26 - $c$ were opposite to one another. Feigin had recently introduced the semi-infinite cohomology of  Lie algebras into mathematics \cite{fml}. For the Virasoro Lie algebra, the obstruction to taking the semi-infinite cohomology of a module vanished only at central charge $26$, and it was already understood in {\em loc. cit.}  that the duality of Verma embeddings at complementary central charges could be recovered from properties of the semi-infinite cohomology functor.   

The Virasoro vertex algebra can be realized as the quantum Hamiltonian reduction of the vacuum vertex algebra for affine $\fsl_2$. Applying the same procedure for other simple Lie algebras $\mathfrak{g}$ produces the $\mathscr{W}$-algebras. A natural conjecture, which appeared in the conformal field theory literature, is that a version of Feigin--Fuchs duality for Verma modules should persist in this wider setting \cite{dv}. Moreover, it was understood that this again should follow from having a theory of semi-infinite cohomology for $\cW$-algebras. 

However, generalizing from $\fsl_2$ to general $\fg$ proved to be far from automatic. Since the operator product expansions of the standard generating fields of a general $\cW$-algebra are not linear in the fields, unlike the case of the Virasoro vertex algebra, the previous Lie algebra formalism for semi-infinite cohomology was inapplicable. Nonetheless, it was clear such a process should exist and play a similarly basic role, e.g. in cancelling unphysical states against ghosts in $\cW$-gravity and $\cW$-string theory \cite{bw3}, \cite{b}, \cite{pope}. Despite constructions of such a cohomology theory and their detailed study for other small rank examples, no general method was found.  

In this paper we construct a semi-infinite cohomology theory for $\cW$-algebras. To do so, we will make essential use of several recent breakthroughs in the local quantum geometric Langlands program. Having done so, we will, as conformal field theorists knew all along, see an appropriate form of Feigin--Fuchs duality for the $\cW$-algebra. Along the way, we will prove several basic results on Category $\OO$ for the $\cW$-algebra, notably a long expected linkage principle.

\section{Statement of Results}\label{two} Let $\fg$ be a simple complex Lie algebra with triangular decomposition $\fg = \fn^- \oplus \h \oplus \fn$.\footnote{All results we prove have analogues for the Lie algebra of a general reductive group, which may be deduced from the simple case. Similarly, one may replace our ground field $\mathbb{C}$ with any algebraically closed field of characteristic zero.}  Let $G$ be the simple, simply connected algebraic group with Lie algebra $\fg$. Fix a level $\kappa$, i.e. an invariant bilinear form $\kappa \in (\fg^* \otimes \fg^*)^G$, and write $\kappa_c$ for the critical level. Write $\gk$ for the affine Lie algebra associated to $\fg$ and $\kappa$, and let $\cw_\kappa$ denote the $\cW$-algebra associated to $\gk$ and a principal nilpotent element of $\fg$. 

\subsection{Categorical Feigin--Fuchs duality} We prove a version of Feigin--Fuchs duality which applies to all representations of the $\cW$-algebra. Since this uses several modern notions from homological algebra, which may not be familiar to all readers, we will approach it via several more elementary statements. A reader comfortable with the basics of stable cocomplete $\infty$-categories may wish to skip directly to Theorem \ref{mainthe}. 

As mentioned in the introduction, Feigin and Fuchs found a remarkable duality between the Verma modules for Virasoro at complementary central charges, i.e. a contravariant equivalence of categories. We show this phenomenon persists for $\cW$. Let us write $\cw_\kappa\operatorname{-mod}^\heartsuit$ for its abelian category of representations, and $\operatorname{Verma}^\heartsuit_\kappa$ for its full subcategory consisting of Verma modules. Then we have:

\begin{theo} There is a canonical equivalence of categories: $$\operatorname{Verma}_\kappa^{\heartsuit, op} \simeq \operatorname{Verma}_{-\kappa + 2\kappa_c}^\heartsuit.$$ \label{app1}
\end{theo}
To our knowledge, Theorem \ref{app1} is new for all cases besides the Virasoro algebra. A next approximation to our Feigin--Fuchs duality is roughly the statement that not only $\operatorname{Hom}$s between Vermas match, but also $\operatorname{Ext}$s. Let us write $\cw_\kappa\operatorname{-mod}^b$ for its bounded derived category of representations, or better its canonical dg-enhancement. If we consider $\operatorname{Verma}_\kappa$, the full pretriangulated subcategory generated by  $\operatorname{Verma}_\kappa^\heartsuit$, we have:

\begin{theo} There is a canonical equivalence of dg-categories: $$\operatorname{Verma}_\kappa^{op} \simeq \operatorname{Verma}_{-\kappa + 2\kappa_c}. $$\label{app22}\end{theo}

Another approximation of Feigin--Fuchs duality is the statement that, as Verma modules are protypical objects of tame ramification, the above admits variants with wild ramification. To state this cleanly involves a slightly subtle point, which we now review. Let us write $\cw_\kappa\operatorname{-mod}^{n}$ ($n$ for naive) for the usual unbounded derived category of $\cw_\kappa$ modules. As objects of $\cw_\kappa\operatorname{-mod}^{n}$, the Verma modules are not compact. However, there is a slight enlargement of $\cw_\kappa\operatorname{-mod}^{n}$, which we denote by $\cw_\kappa\operatorname{-mod}$, in which the Verma modules are now compact. An analogous renormalization was first introduced for $\gk$ by Frenkel--Gaitsgory, and later for $\cw_\kappa$ by Raskin \cite{fg}, \cite{r}.\footnote{As a remark for experts - we in fact show all of $\gk\operatorname{-mod}^n$ embeds fully faithfully into $\gk\operatorname{-mod}$, cf. Proposition \ref{basl}. Previously this was only known for the bounded below derived category, and the same proof applies to  $\cw\operatorname{-mod}$ as well.} Such renormalizations are now more broadly seen as useful facts of life in both singular finite dimensional and smooth infinite dimensional algebraic geometry, and quantizations thereof \cite{gi}, \cite{gn}. 

Within $\cw_\kappa\operatorname{-mod}$ one has its subcategory of compact objects $\cw_\kappa\operatorname{-mod}^c$. We now give a heuristic sense of what these look like. Recall that $Z \fg$, the center of the universal enveloping algebra of $\fg$, quantizes the invariant theory quotient $\fg/\!\!/G$. Similarly, the topological associative algebra associated to $\cW$ quantizes the loop space $L(\fg/\!\!/G)$. Bounding the order of poles of an algebraic loop exhibits $L(\fg/\!\!/G)$ as an ascending union of pro-finite-dimensional affine spaces, and objects of $\cw_\kappa\operatorname{-mod}^c$  quantize pushforwards of perfect complexes from these subschemes.\footnote{In modern parlance, the latter are $\operatorname{IndCoh}(L(\fg/\!\!/G))^c$.}  

With these explanations in place, we can state:

\begin{theo}There is a canonical equivalence of dg-categories:

\begin{equation}\cw_\kappa\operatorname{-mod}^{c, op} \simeq \cw_{-\kappa + 2\kappa_c}\operatorname{-mod}^c.\label{app3eq}\end{equation}
\label{app3}
\end{theo}

In the case of the Virasoro algebra, it had long been a folklore conjecture that something like \eqref{app3eq} should be true. Positselski has given a solution in the interesting monograph \cite{pos} for any Tate Lie algebra. His formulation, and in particular the types of exotic derived categories he works with, are seemingly different from ours. We also prove a Tate Lie algebra analogue of Theorem \ref{app3} in Theorem \ref{tld}, and it would be good to understand the precise relationship between this and {\em loc. cit.}.  However, for all other $\cW$-algebras we unaware of previous work along these lines. 

Finally, we explain a reformulation of Theorem \ref{app3} that applies to the entirety of $\cw_\kappa\operatorname{-mod}$. The totality of all $\mathbb{C}$-linear cocomplete dg-categories can be organized into an $(\infty, 2)$-category $\operatorname{DGCat}_{cont}$, whose homotopy classes of 1-morphisms are $\mathbb{C}$-linear continuous quasifunctors. This is further a symmetric monoidal $\infty$-category, so that given two cocomplete dg-categories $\scc, \mathscr{D}$, one can form their tensor product $\scc \otimes \mathscr{D}$. In particular, one can make sense of dual objects $\scc, \scc^\vee$ in $\operatorname{DGCat}_{cont}$, and Theorem \ref{app3} is equivalent to:

\begin{theo} There is a canonical duality of cocomplete dg-categories: \begin{equation}\cw_{\kappa}\operatorname{-mod}^\vee {\simeq} \cw_{-\kappa + 2\kappa_c}\operatorname{-mod}.\label{maineq}\end{equation} \label{mainthe} \end{theo}
From a duality datum, one obtains a contravariant equivalence between  the compact objects, recovering Theorem \ref{app3}, and we show this sends Verma modules to Verma modules, recovering Theorems \ref{app1}, \ref{app22}. Moreover, as dual categories, representations of complementary $\cW$-algebras have a `perfect pairing'$$\cw_\kappa\operatorname{-mod} \otimes \cw_{-\kappa + 2\kappa_c}\operatorname{-mod} \rightarrow \operatorname{Vect},$$which is the promised semi-infinite cohomology functor. The problem of constructing such a functor was first raised in the conformal field theory and string theory literature, and solved in several small rank cases   \cite{p2}, \cite{p3}, \cite{p1}, \cite{bw3}, \cite{b},  \cite{p5}, \cite{p7}, \cite{p8}, \cite{sd}, \cite{sd2}, \cite{p4}, \cite{p6}. Its existence has also been anticipated in the mathematical literature. Notably, I. Frenkel and collaborators have developed a rich program on the relation between quantum groups and affine Lie algebras that links both through $\cW$. A crucial role is played by remarkable conjectural calculations of $\cW$ semi-infinite cohomology, e.g. a putative construction of algebras of functions on quantum groups via the modified regular representations of $\cW$-algebras \cite{f1}, \cite{f2}, \cite{f5}, \cite{f4}, \cite{f3}. Finally, we would like to mention that the kernel realizing the duality is closely related to the chiral universal centralizer, i.e. the bi-Whittaker reduction of chiral differential operators on the group. This remarkable vertex algebra, which was studied for $\fg = \fsl_2$ in work of Frenkel--Styrkas \cite{f1}, prominently figures in the recent breakthroughs of Arakawa on the 2d/4d correspondence of Beem {\em et. al.} for genus zero theories of Class $\mathcal{S}$ and the Moore--Tachikawa conjecture \cite{araclasss}, \cite{beem}, \cite{ga},  \cite{mta}.

\subsection{Homomorphisms between Verma modules and the linkage principle}  Feigin and Fuchs not only proved a contravariant equivalence as in Theorem \ref{app1} for the Virasoro algebra, but also explicitly determined all the morphisms in the category $\operatorname{Verma}_\kappa^\heartsuit$. We will presently describe a similar classification for a general $\mathscr{W}$-algebra. Recall that highest weights for $\cW_\kappa$, i.e. the maximal spectrum of its Zhu algebra, identify with $\Wf \backslash \h^*$, where $\Wf$ denotes the finite Weyl group, cf. Subsection \ref{zhu} for our normalizations.

Fix a noncritical level $\kappa$. To describe $\operatorname{Verma}_\kappa^\heartsuit$, we first determine the block decomposition of Category $\OO$. I.e., we determine the finest partition $$\Wf \backslash \h^* = \bigsqcup_{\alpha} \mathscr{P}_\alpha$$such that every module $M$ in $\OO$ decomposes as a direct sum of submodules $\bigoplus_\alpha M_\alpha$ wherein all the simple subquotients of $M_\alpha$ have highest weights in $\mathscr{P}_\alpha$.

Loosely, we find that blocks for $\cW_\kappa$ are projections of blocks for Kac--Moody. More precisely, consider the level $\kappa$ dot action of the affine Weyl group $W$ on $\h^*$ and the projection $$\pi: \h^* \rightarrow \Wf \backslash \h^*.$$ For $\lambda \in \h^*$, write $W_\lambda$ for its integral Weyl group, $W_{\operatorname{f}, \lambda}$ for the intersection $W_\lambda \cap \Wf$, and $W_\lambda^\circ$ for the stabilizer of $\lambda$ in $W_\lambda$. 

\begin{theo} (Linkage principle) The block of $\OO$ containing the simple module with highest weight $\pi(\lambda)$ has highest weights $$\pi( W_\lambda \cdot \lambda) \simeq W_{\operatorname{f}, \lambda} \backslash W_\lambda / W_\lambda^\circ.$$
\label{linkit}
\end{theo}

 Theorem \ref{linkit} has been anticipated as a folklore conjecture since the earliest days of the subject, cf. \cite{bw3}. To our knowledge, this is its first appearance for all cases save the Virasoro algebra. It was also found by Arakawa (unpublished). Using Theorem \ref{linkit}, we formulate in Section 8 several conjectures on the structure of blocks and a conjectural relation between Drinfeld--Sokolov reduction and translation functors.  

To describe the homomorphisms between Verma modules within a single block, recall that since $\kappa$ is noncritical, $W_\lambda \cdot \lambda$ always contains an antidominant or dominant weight. Specifically, if $\kappa$ is negative every block contains an antidominant weight, and if $\kappa$ is positive every block contains a dominant one. Thus we may assume that $\lambda$ is (anti-)dominant, in which case $W_{\operatorname{f}, \lambda}$ and $W_\lambda^\circ$ are parabolic subgroups of $W_\lambda$, and hence $\dcs$ carries a Bruhat order.  

\begin{theo} For $w \in \dcs$ write $M_w$ for the corresponding Verma module. Then, for $y,w \in \dcs$, we have in the abelian category $\cw_\kappa\operatorname{-mod}^\heartsuit$:

\begin{enumerate}
    \item $\Hom(M_y, M_w)$ is at most one dimensional, and any nonzero morphism is an embedding. 
    \item If $\lambda$ is antidominant, then $\Hom(M_y, M_w)$ is nonzero if and only if $y \leqslant w$. \
    \item Suppose $\lambda$ is dominant. Then $\Hom(M_y, M_w)$ is nonzero if and only if $y \geqslant w$. 
\end{enumerate}

\label{m2v}
\end{theo}
Theorem \ref{m2v} was conjectured in the conformal field theory literature by de Vost--van Driel in the remarkable paper \cite{dv}. The main non-elementary input into Theorem \ref{linkit} and the negative level cases of Theorem \ref{m2v} is Arakawa's resolution of the Frenkel--Kac--Wakimoto conjecture on Drinfeld--Sokolov reduction \cite{ara}, \cite{fkw}. It is unclear whether Theorem \ref{m2v}(1) at positive level can be proven by similar methods. However, following the fundamental insight of Gaitsgory and his school that categorical dualities can see concrete representation-theoretic phenomena, we instead use Categorical Feigin--Fuchs duality to reduce to the negative level case.  

\subsection{Future directions} The constructions and applications described in this paper mostly concern the algebraic representation theory of $\cW$-algebras. However, we would like to mention two further places where its contents are relevant that have a slightly different flavor, namely the geometric representation theory of $\cW$-algebras and applications to the quantum Langlands correspondence. 

First, a basic problem, still unsolved for the Virasoro algebra, has been to produce a localization theorem for highest weight representations of $\cW$-algebras analogous to those available for simple and affine Lie algebras. As we will explain elsewhere, this can be done using Whittaker sheaves on the enhanced affine flag variety. Moreover, to explicitly identify the representations realized in this way one needs to speak of blocks of Category $\OO$ for the $\cW$-algebra. This motivated our determination of the linkage principle in the current work. Further, at positive level, the localization is essentially of derived nature.\footnote{In particular, a version of Kashiwara--Tanisaki localization, which has not yet appeared in the literature, holds for $\gk$ at positive level. } This motivated the present study of Categorical Feigin--Fuchs duality, which may be used  to reduce to the $t$-exact negative level case.

Second, the closely related category of Whittaker sheaves on the affine Grassmannian is the subject of the conjectural Fundamental Local Equivalence (FLE) of the quantum geometric Langlands program \cite{gtw}, \cite{gfl}. This conjecture, due to Gaitsgory--Lurie, provides a remarkable deformation of the Geometric Satake isomorphism to all Kac--Moody levels. Conjectures of Aganacic--Frenkel--Okounkov suggest that the FLE should be provable using the representation theory of $\cW$-algebras \cite{afo}. We can currently give such a proof of the FLE over a point, i.e. non-factorizably, for $\fg = \fsl_2$, and the case of general $\fg$ over a point is work in progress with Raskin. The above localization theorem and the results of this paper play a basic role. Similar arguments, which are work in progress, apply to a tamely ramified variant of the FLE conjectured by Gaitsgory \cite{gkl}; for appropriately integral levels this has been implemented in a forthcoming work with Campbell \cite{tr}.

\subsection{Organization of the Paper} In Section 3, we collect basic definitions and notations. In Section 4, as preparation for Feigin--Fuchs duality, we prove a general duality statement for Whittaker models of categorical loop group representations. In Section 5, we derive Feigin--Fuchs duality for $\cW$-algebras and Tate Lie algebras. In Section 6, we prove some basic structural properties of Category $\OO$ for $\cw_\kappa$. In Sections 7 and 8, we prove the linkage principle, and propose a conjectural description of the blocks for the $\cW$-algebra. Finally, in Section 9 we prove de Vost--van Driel's conjecture on Verma embeddings.

\label{sres}
\subsection*{Acknowledgments} We thank Tomoyuki Arakawa, Dima Arinkin, Christopher Beem, Roman Bezrukavnikov, Alexander Braverman, Dan Bump, Thomas Creutzig, Davide Gaiotto, Dennis Gaitsgory, Alexander Goncharov, Edward Frenkel, Igor Frenkel, Flor Hunziker, Victor Kac, Sam Raskin, Ben Webster, Emilie Wiesner, David Yang, Zhiwei Yun, and Gregg Zuckerman for helpful discussions and correspondence.

\section{Preliminaries}
In this section we establish notation and collect some facts which will be useful to us. The reader may wish to skip to the next section and refer back only as needed.

Our notation and conventions are standard, with the mild exceptions of our normalizations of the isomorphism $\operatorname{Zhu}(\cW_\kappa) \simeq Z \fg$ and the duality on Category $\OO$ for $\cW_\kappa$, cf. Remark \ref{zhu}, and Subsection \ref{dp} respectively.

\subsection{The affine Kac--Moody algebra} Recall that $\fg = \fn^- \oplus \h \oplus \fn$ is a simple Lie algebra. Let $\hat{\fg}_{AKM} = \fg[z, z^{-1}] \oplus \C c \oplus \C D$ be the affine Kac--Moody algebra associated to $\fg$, cf. \cite{kitty}. We use the standard normalizations, so that $$[X \otimes z^n, Y \otimes z^m] = [X,Y] \otimes z^{n+m} + n\delta_{n, -m}  \kappa_b(X,Y)c, \quad \quad X, Y \in \fg, \quad n,m \in \mathbb{Z},$$where $\kappa_b$ is the basic invariant inner product, i.e. normalized so that $\kappa_b( \check{\theta}, \check{\theta}) = 2,$ where $\check{\theta} \in \h$ is the coroot associated with the highest root $\theta \in \h^*$ of $\fg$. The remaining brackets are: $$[\C c, \hat{\fg}_{AKM}] = 0, \quad \quad [D, X \otimes z^n] = n X \otimes z^n, \quad \quad X \in \fg.$$
The affine Kac--Moody algebra has a standard triangular decomposition $\hat{\fg} = \hat{\fn}^- \oplus \hat{\h} \oplus \hat{\fn}$, where $\hat{\h} = \h \oplus \C c \oplus \C D$. Write $\check{\alpha}_i \in \h, \alpha_i \in \h^*, i \in I,$ for the simple coroots and roots of $\mathfrak{g}$. Then the simple coroots and roots for $\hat{\fg}_{AKM}$ are naturally indexed by $\hat{I} := I \sqcup \{0\}$. Explicitly, the simple coroots are given by: $$\check{\alpha}_i \in \h \oplus 0 \oplus 0, i \in I,  \quad \operatorname{and} \quad  \check{\alpha}_0 := - \check{\theta} + c \in \h \oplus \C c \oplus 0.$$ To write down the simple roots, we decompose $$\hat{\h}^* \simeq \h^* \oplus (\C c)^* \oplus (\C D)^* = \h^* \oplus \C c^* \oplus \C D^*,$$where $\langle c^*, c \rangle = 1$, $\langle D^*, D \rangle = 1$. With this, the simple roots are given by: $$\alpha_i  \in \h^* \oplus 0 \oplus 0, i \in I, \quad \operatorname{and} \quad \alpha_0 := -\theta + D^* \in \h^* \oplus 0 \oplus \C D^*.$$  

\subsection{The affine Weyl group and the level $\kappa$ action}
For each index $i \in \hat{I}$, we have an associated simple reflection $$s_i: \hat{\h}^* \rightarrow \hat{\h}^*, \quad s_i(\lambda) := \lambda - \langle \lambda, \check{\alpha}_i \rangle \alpha_i.$$The subgroup of $GL(\hat{\h}^*)$ generated by these reflections $s_i, i \in \hat{I},$ is the affine Weyl group $W$. The subgroup generated by $s_i, i \in I,$ is the  finite Weyl group $W_{\operatorname{f}}$ associated to $\fg$. In both cases, these preferred generators exhibit $W$ and $W_{\operatorname{f}}$ as Coxeter groups.

By construction, $W$ acts trivially on $\C D^*$.
It follows that $W$ acts linearly on the quotient $\hat{\h}^*/ \C D^* \simeq \h^* \oplus \C c^*$. The action of $W$ on $\h^* \oplus \C c^*$ further preserves the affine hyperplanes $H_k$ cut out by $c = k, k \in \C$.\footnote{Of course, the same holds without quotienting by $\C D^*$.} Therefore, under the affine linear isomorphism $\h^* \simeq H_k,$ sending $\lambda \in \h^*$ to $\lambda + k c^*$, we obtain an action of $W$ on $\h$ by affine linear automorphisms, which we call {\em the level $k$ action}. 

Explicitly, at level $k$, $W_{\operatorname{f}}$ acts in the usual way on $\h^*$, and $s_0$ acts by the affine reflection through $\check{\theta} = k$:$$s_0(\lambda) = \lambda - (\langle \lambda, \check{\theta} \rangle  - k) \theta.$$By composing $s_0$ with the reflection $s_\theta \in W_{\operatorname{f}}$ associated with root $\theta$, one obtains translation by $k \theta$. Using this, one finds:

\begin{pro} Write $\Lambda \subset \h^*$ for the lattice generated by the long roots. Then the level $k$ action of $W$ on $\h^*$ identifies with $W_{\operatorname{f}} \ltimes \Lambda$, where $W_{\operatorname{f}}$ acts by usual reflections and $\lambda \in \Lambda$ acts by the dilated translation $t_{k, \lambda}(\nu) = \nu + k \lambda, \nu \in \h^*$. 
\label{kact} \end{pro}

\label{wgp}
\subsection{The affine Lie algebra at level $\kappa$ and the level $\kappa$ action} 
\label{sss}
Recall that to any invariant inner product $\kappa$ on $\fg$, we have associated central extension of $\fg(\!(z)\!)$ $$0 \rightarrow \C \mathbf{1} \rightarrow \gk \rightarrow \fg(\!(z)\!) \rightarrow 0,$$with commutator given by: $$[X \otimes f, Y \otimes g] = [X,Y] \otimes fg + \kappa(X,Y) \operatorname{Res}_{z = 0} f dg, \quad \quad X,Y \in \fg, \quad f,g \in \mathbb{C}(\!(z)\!).$$Since $\fg$ is simple, we may write $\kappa = k \kappa_b, k \in \C$.

Write $\hat{\fg}_{AKM} \operatorname{-mod}_k^\heartsuit$ for the abelian category of smooth representations of $\hat{\fg}_{AKM}$ on which $c$ acts by the scalar $k \in \mathbb{C}$. Similarly, write $\gk\operatorname{-mod}^\heartsuit$ for the abelian category of smooth representations of $\gk$ on which $\mathbf{1}$ acts by the identity. By smoothness, restricting an action of $\hat{\fg}_{AKM}$ to $\fg[z, z^{-1}] \oplus \C c$ produces a functor $\hat{\fg}_{AKM} \operatorname{-mod}_k^\heartsuit \rightarrow \gk\operatorname{-mod}^\heartsuit$. When $\kappa \neq \kappa_c$, this functor admits a distinguished section, given by the Segal--Sugawara construction, cf. \cite{Fr}. With these identifications, we have: 

\begin{pro} Let $\kappa = k \kappa_b, \kappa \neq \kappa_c$. Then the level $k$ action of $W$ on $\h^*$ considered in Subsection \ref{wgp} coincides with the standard action of $W$ on the weights of $\h$-integrable $\gk$ modules. 
\end{pro}
Henceforth, we will speak interchangeably of the level $k$ and $\kappa$ actions. 

\subsection{Combinatorial duality}\label{combd} Having explained the level $\kappa$ action, we presently discuss a combinatorial shadow of the duality between $\gk\operatorname{-mod}$ and $\hat{\fg}_{-\kappa + 2 \kappa_c}\operatorname{-mod}$. From Proposition \ref{kact}, $W$ acts at level $\kappa$ through $W_{\operatorname{f}} \ltimes \Lambda$, where the translation action of $\Lambda$ is dilated by $\frac{\kappa}{\kappa_b}$. It follows that the orbits for the level $\kappa$ and $\kappa'$ actions coincide only when $\kappa + \kappa' = 0$. In this case, it is not quite the case that level $\kappa$ and $\kappa'$ actions coincide, since $s_0$ acts as the affine reflection through $\check{\theta} = \pm \frac{\kappa}{\kappa_b}$. Incorporating this sign, we have:
\begin{pro} Write $\h^*_{\kappa, naive}$ for $\h^*$ with the level $\kappa$ action of $W$. Then negation $\lambda \rightarrow - \lambda$ is a W-equivariant isomorphism  $\h^*_{\kappa, naive} \simeq \h^*_{-\kappa, naive}.$ \label{dnaive}
\end{pro}
The appearance of the critical twist comes, as usual, from a $\rho$ shift. To explain this, write $\rho \in \h^*$ for the unique element satisfying $\langle \rho, \check{\alpha}_i \rangle = 1, i \in I$. Write $\hat{\rho} \in \hat{\h}^*$ for an element satisfying $\langle \hat{\rho}, \check{\alpha}_i \rangle  = 1, i \in \hat{I}.$ Explicitly, the possible $\hat{\rho}$ form the affine line $\rho + h^\vee c^*  + \mathbb{C} D^*$, where $h^\vee$ is the dual Coxeter number. Recalling that $\kappa_c = - h^{\vee} \kappa_b$, in particular any $\hat{\rho}$ will lie in the affine hyperplane of $\hat{\h}^*$ corresponding to minus the critical level. 

We then have the dot action of $W$ on $\hat{\h}^*$, where $w \cdot \lambda = w(\lambda + \hat{\rho}) - \hat{\rho}$, which is independent of the choice of $\hat{\rho}$. This is the usual action, albeit centered at $- \hat{\rho}$ rather than 0. That is, the map $\hat{\h}^* \rightarrow \hat{\h}^*$ sending $\lambda$ to $\lambda - \hat{\rho}$ intertwines the usual action of $W$ on the domain and the dot action of $W$ on the target.

As in Subsection \ref{wgp}, the dot action descends to an action on $\h^* \oplus \C c^*$. This action preserves each affine hyperplane $c = \frac{\kappa}{\kappa_b}$, which gives us a level $\kappa$ dot action on $\h^* \simeq \h^* + \frac{ \kappa}{\kappa_b} c^* \subset \h^* \oplus \C c^*,$ which we denote by $\h^*_{\kappa}$. By construction, we have an isomorphism \begin{equation} \label{ps} \iota: \h_{\kappa, naive}^* \simeq \h_{\kappa + \kappa_c}^* \quad \quad \iota(\lambda) =\lambda - \rho, \quad \lambda \in \h^*_{\kappa, naive}. \end{equation}Concatenating Equation \eqref{ps} and Proposition \ref{dnaive} gives:

\begin{pro} There is a $W$ equivariant isomorphism: \begin{equation} i: \h^*_{\kappa} \simeq \h^*_{-\kappa + 2 \kappa_c}, \quad \quad i(\lambda) = - \lambda - 2 \rho, \quad \lambda \in \h^*_{\kappa}. \end{equation}
\label{flippy}
\end{pro}
As we will see, this naive duality between orbits at positive and negative levels will show up in the combinatorics of the tamely ramified cases of Kac--Moody duality, in particular the duality of Verma modules, cf. Theorem \ref{vermsd}. 

\subsection{The $\mathscr{W}$-algebra} Some nice references for $\cW$-algebras are \cite{aran}, \cite{fbz}. Recall we wrote $\fg = \fn^- \oplus \h \oplus \fn$. For a complex Lie algebra $\mathfrak{a}$, write $L\mathfrak{a}$ for the topological Lie algebra $\mathfrak{a}(\!(z)\!)$, and similarly $L^+ \mathfrak{a} := \mathfrak{a}[[z]]$. \label{defw}

Fix a nondegenerate character $\psi: \fn \rightarrow \mathbb{C}$, i.e. one which is nonzero on each simple root subspace of $\fn$. Still write $\psi$ for the character of $L\fn$ given by the composition $$L\fn \xrightarrow{\operatorname{Res}} \fn \xrightarrow{\psi} \mathbb{C},$$which depends on our choice of coordinate $z$. Writing $\mathbb{C}_\psi$ for the associated character representation of $L\fn$, and $\mathbb{V}_\kappa$ for the vacuum algebra for $\gk$, cf. \cite{Fr}, \cite{fz}, we have \begin{equation}\cW_\kappa := H^{\frac{\infty}{2} + 0}( L\fn, L^+\fn, \mathbb{V}_\kappa \otimes \mathbb{C}_{\psi}).\label{brst}\end{equation}
$\cW_\kappa$ is a vertex algebra, and for $\kappa$ noncritical contains a canonical conformal vector $\omega$. Write $L_0$ for the corresponding energy operator, i.e. the coefficient of $z^{-2}$ in the field $\omega(z)$. Under the adjoint action of $L_0$, $\cW_\kappa$ acquires a $\mathbb{Z}^{\geqslant 0}$ grading. We say $a \in \cW_\kappa$ is homogeneous of degree $\lvert a \rvert \in \mathbb{Z}^{\geqslant 0}$ if $L_0 a = \lvert a \rvert a $.

\label{dsred}

\subsection{Category $\OO$} In what follows, we only consider levels $\kappa \neq \kappa_c$. Consider the abelian category $\mathscr{W}_\kappa\operatorname{-mod}^\heartsuit$ of  modules for $\mathscr{W}_\kappa$, cf. \cite{fbz}. The following definition is more or less due to Arakawa in \cite{ara}, see however Remark \ref{ssss}. 

\begin{defn} $\OO$ is the full subcategory of $\mathscr{W}\operatorname{-mod}^\heartsuit$ consisting of objects $M$ satisfying:

\begin{enumerate}
    \item $M$ decomposes as a sum of generalized eigenspaces under the action of $L_0$, i.e. we may write $$M = \bigoplus_{d \in \C} M_d, \quad \quad M_d := \{ m \in M: (L_0 + d)^N m = 0, N \gg 0 \}.$$ 
    \item For each $d \in \C$, $M_d$ is finite dimensional, and $M_{d + n} = 0$ for all $n \in \Z^{\geqslant N}, N \gg 0.$ 
\end{enumerate}
\label{defoglarg}

\end{defn}

\begin{re} Allowing generalized eigenspaces of $L_0$ may be compared with allowing generalized eigenvalues of the Cartan subalgebra and generalized central characters in the formation of Category $\OO$ for a semisimple Lie algebra, i.e. we obtain the `free-monodromic' Category $\OO$. To our knowledge, a good definition of the `non-monodromic' category for $\cW$-algebras other than Virasoro is unknown.\end{re}

\begin{re} We should mention that since $\kappa \neq \kappa_c$, we do not consider the grading as an auxiliary structure, but rather induced from the conformal vector of $\mathscr{W}$. This differs from the treatment of Arakawa \cite{ara}, as in the discussion of Subsection \ref{sss} for Kac--Moody.\label{ssss} \end{re}
\label{cato}

Let $M$ be an object of $\OO$. For $m \in M$, recall that $m$ is said to be {\em singular} if, for every homogeneous $a \in \mathscr{W}_\kappa$, $a(z)m \in M(\!(z)\!)$ has a pole of at most order $\lvert a \rvert$. Write $S(M)$ for the subspace of all singular vectors in $M$. This naturally carries an action of the Zhu algebra $\operatorname{Zhu}(\mathscr{W}_\kappa)$. The latter can be identified with $Z\fg$, the center of the universal enveloping algebra of $\fg$. We will explain our precise normalization for this isomorphism in Remark \ref{zhu} below, but in particular by the Harish--Chandra isomorphism we may identify characters of $\operatorname{Zhu}( \mathscr{W}_\kappa)$ with $\Wf \backslash \h^*$, the quotient of $\h^*$ by the dot action of $\Wf$. 

If $m \in S(M)$ is an eigenvector for the action of  $\operatorname{Zhu}(\mathscr{W})$, with eigenvalue $\lambda \in W_{\operatorname{f}} \backslash \h^*$, we say $m$ is a {\em highest weight vector} with {\em highest weight} $\lambda$. Morphisms in $\OO$ preserve singular and highest weight vectors. Accordingly, for $\lambda \in W_{\operatorname{f}} \backslash \h^*,$ there is an associated functor $m_\lambda: \OO \rightarrow \operatorname{Vect}$, sending an object $M$ to its highest weight vectors of weight $\lambda$. This is corepresented by an object $M(\lambda)$, called the Verma module. $M(\lambda)$ has a unique simple quotient $L(\lambda)$, and every simple object of $\OO$ is of the form $L(\lambda)$ for a unique $\lambda \in W_{\operatorname{f}} \backslash \h^*$.  

We say a module $M$ is a {\em highest weight module} if it can be generated by a highest weight vector. Equivalently, $M$ admits a surjection from $M(\lambda)$, for some necessarily unique $\lambda \in W_{\operatorname{f}} \backslash \h^*$. 

Feeding representations of $\gk\operatorname{-mod}^\heartsuit$ into a complex similar to Equation \eqref{brst} produces representations of $\cW_\kappa$. The following fundamental theorem of Arakawa informally says that this relates Category $\OO$ for $\gk$ and $\cW_\kappa$ in an extremely tight way.
 
\begin{theo} \cite{ara} Write $\OO_{\gk}$ for the Category $\OO$ associated to $\gk$. For $\lambda \in \h^*$, write $M(\lambda)$ and $L(\lambda)$ for the associated Verma and simple modules for $\gk$, and recall the projection $\pi: \h^* \rightarrow \Wf \backslash \h^*$. There exists an exact functor $$\operatorname{DS}_-: \OO_{\gk} \rightarrow \OO$$such that for any $\lambda \in \h^*$, $\operatorname{DS}_-(M(\lambda)) \simeq M(\pi(\lambda))$ and

$$\operatorname{DS}_{-}(L(\lambda)) \simeq \begin{cases} L(\pi(\lambda)) & \langle \lambda + \hat{\rho}, \halpha_i \rangle \notin \mathbb{Z}^{> 0}, i \in I, \\ 0 & \text{otherwise.} \end{cases}.$$ \label{mrd}
\end{theo}

\begin{re}In particular, we normalize the isomorphism $\operatorname{Zhu}(\cW_\kappa) \simeq Z\fg$ by composing the isomorphism $\operatorname{Zhu}(\cW_\kappa) \simeq Z \fg$ used by Arakawa in {\em loc. cit.} with the involution $Z \fg \simeq Z \fg$ corresponding to the automorphism of $\Wf \backslash \h^*$ given by $\pi(\lambda) \mapsto \pi( - w_\circ \lambda)$, where $w_\circ$ denotes the longest element of $\Wf$.\label{zhu} \end{re}

\subsection{Duality}\label{dp}
 In {\em loc. cit.}, Arakawa shows that  $\OO$ admits a contravariant equivalence $\mathbf{D}': \OO \rightarrow \OO^{op}$. By definition, the underlying vector space of $\mathbf{D}' M$ is $\bigoplus_{d \in \C} M_d^*$, and the action of $\mathscr{W}$ is given by the adjoint action precomposed by an appropriate Chevalley anti-involution. 
 
 In {\em loc. cit.}, it is shown that $\mathbf{D}'  L(\pi(\lambda)) \simeq L(\pi(-w_\circ \lambda))$. We may modify it  to remove the appearance of $w_\circ$ as follows. Fix Chevalley generators $e_i, h_i, f_i, i \in I$ for $\fg$. The longest element $w_\circ$ of $\Wf$ sends the positive simple roots to the simple negative roots. I.e., we obtain an involution $\tau$ of $I$ such that $$\operatorname{Ad}_{w_\circ} \mathbb{C} e_i ~= \mathbb{C} f_{\tau(i)}, i \in I.$$Write $\tau$ for the involution of $\fg$ sending $e_i, h_i, f_i$ to $e_{\tau( i)}, h_{\tau(i)}, f_{\tau(i)}, i \in I$, respectively. This induces involutions $\tau$ of $L^+ \fg$ and $\gk$,  and hence $\mathbb{V}_k$. Similarly, it induces involutions of $L\fn$ and $L^+\fn,$  and if we pick $\psi$ so that $\psi(e_i) = 1, i \in I,$ we obtain an involution $\tau$ of  $\csi(L\fn, L^+ \fn, \mathbb{V}_k \otimes \mathbb{C}_\psi)$, and hence of $\cW_\kappa$. Writing $\tau^*$ for the restriction functor $\cw_\kappa\operatorname{-mod}^\heartsuit \rightarrow \cw_\kappa\operatorname{-mod}^\heartsuit$ induced by $\tau$, we define $\mathbf{D} ~:= \tau^* \circ \mathbf{D}'$. With this, we have
 $$\mathbf{D} L(\lambda) \simeq L(\lambda), \quad \quad \lambda \in \Wf \backslash \h^*.$$

Let us say that the module $A$ is {\em co-highest weight} if $\mathbf{D} A$ is a highest weight module. Calling $A(\lambda) := \mathbf{D} M(\lambda)$ a {\em co-Verma} module, equivalently a module $A$ is {\em co-highest weight} if it admits an injection into a necessarily unique $\mathbf{D} M(\lambda)$, for some $\lambda \in W_{\operatorname{f}} \backslash \h^*$.

\subsection{$q$-characters} The following will only be used in the proof of Lemma \ref{vcv}. For an object $M = \bigoplus_{d \in \C} M_d$ of $\OO$, its $q$-character is the formal series $$\ch_q M := \sum_{d \in \C} (\dim M_{-d}) q^d.$$

The basic properties which will be of use to us are summarized in the following

\begin{pro} For any object $M$ of $\OO$, we have $$\ch_q M = \ch_q \mathbf{D} M.$$In particular, $\ch_q M(\lambda) = \ch_q A(\lambda), \lambda \in \Wf \backslash \h^*$.  \label{chardual}
\end{pro}

\section{Duality for Whittaker models of categorical loop group representations}
\label{dwm}

In this section, we use the adolescent Whittaker filtration of Raskin to construct a duality for Whittaker models of strong categorical representations of loop groups. A very readable reference for background on categorical representations of groups and loop groups, and in particular $D$-modules on loop groups, is \cite{Beraldot}.

Let us mention that this duality is already known to experts. Namely, the commutation of Whittaker invariants with tensoring by an object of $\operatorname{DGCat}_{cont}$, viewed as a trivial strong representation, already formally implies the duality of Whittaker models. This in turn follows from knowing that Whittaker invariants and coinvariants can be identified,  cf. \cite{dnotes}, \cite{denniswhit}. However, the present adolescent Whittaker construction of the duality, which has not yet appeared in the literature, gives a transparent picture helpful for concrete applications, in particular the relation to dualities on spherical and baby Whittaker invariants and the action on compact objects. These compatibilities are crucial for representation-theoretic applications in the remainder of the paper.

\subsection{Recollections on the adolescent Whittaker construction} Let $U$ be a quasicompact, pro-unipotent group scheme, and $\mathscr{C}$ a cocomplete dg-category with a strong action of $U$. Recall that we may form the invariants $\mathscr{C}^U$ and the coinvariants $\mathscr{C}_U$, and that the tautological composition \begin{equation} \label{naive} \mathscr{C}^U \xrightarrow{\text{Oblv}} \mathscr{C} \xrightarrow{\text{ins}} \mathscr{C}_U\end{equation}gives an equivalence $\mathscr{C}^U \simeq \mathscr{C}_U$ \cite{Beraldo}. 
\label{adwc}Recall that $G$ was simple with Lie algebra $\fg = \fn^- \oplus \h \oplus \fn$. Write $N$ for the connected unipotent subgroup of $G$ with Lie algebra $\fn$, and $L^+G, L^+N$ for their arc groups and $LG, LN$ for their loop groups, respectively. Let $\mathscr{C}$ be a cocomplete dg-category with a strong action of $LN$. $LN$ is now an  ind-prounipotent group scheme, and the analogous map $\scc^{LN} \rightarrow \scc_{LN}$ need not be an equivalence. However, a theorem of Raskin salvages a twisted version of this statement, as we now remind. For $\psi: LN \rightarrow \mathbb{G}_a$ a character, we may form the invariants $\mathscr{C}^{LN, \psi}$ and coinvariants $\mathscr{C}_{LN, \psi}$. 

The relevant $\psi$ is obtained as follows. Let $\psi: N \rightarrow \mathbb{G}_a$ be a character. Then $\psi$ necessarily factors through $N/[N, N]$. For each simple root $\alpha_i, i \in I$, we have the corresponding one parameter subgroup $U_{\alpha_i} \rightarrow N$, and an isomorphism $$N/[N, N] \simeq \prod_{i \in I} U_{\alpha_i}.$$ Recall that $\psi$ is said to be nondegerate if its restriction to each simple root subgroup $U_\alpha$ is nonzero. Associated to our choice of coordinate $z$ is a residue map $\operatorname{Res}: L\mathbb{G}_a \rightarrow \mathbb{G}_a$. Accordingly, we may form the composition \begin{equation}\label{psi} LN \rightarrow L(N/[N, N]) \xrightarrow{\operatorname{Res}} N/[N, N] \xrightarrow{\psi} \mathbb{G}_a,\end{equation}which we continue to denote by $\psi$. 

\begin{theo} \cite{r} Let $\psi: LN \rightarrow \mathbb{G}_a$ be a character as in Equation \eqref{psi}. Then for any $\kappa$ and $\scc$ with a strong level $\kappa$ action of $LG$, there is a canonical isomorphism \begin{equation}\scc^{LN, \psi}  \simeq \scc_{LN, \psi}.\label{rl}\end{equation} \label{whitteq}
\end{theo}

We emphasize that one asks that the action of $LN$ be the restriction of an action of $LG$, and indeed the functor realizing the equivalence of Theorem \ref{whitteq} is not the analog of \eqref{naive}. Instead, as we recall presently, it is subtler, and uses more of the $LG$ action.  

Consider the decreasing sequence of subgroups of $L^+G:$ $$I'_n := L^+G  \overset{G[z]/z^{n}}{\times} N[z]/z^{n},\quad \quad n \in \Z^{\geqslant 1}.$$In words, $I'_n$ is jets into $G$ which to $(n-1)^{st}$ order lie in $N$. Note that for all $n \geqslant 1,$ $I'_n$ is prounipotent.  Raskin skews the $I'_n$ by the cocharacter $\check{\rho}$ of the adjoint torus $H_{ad}$ to obtain $$I_n := \operatorname{Ad}_{t^{-n \check{\rho}}} I_n, \quad \quad n \in \Z^{\geqslant 1}.$$Using the map $I'_n \rightarrow N[z]/z^n$, $\psi$ induces a homomorphism $I_n \rightarrow \mathbb{G}_a$ which coincides with \eqref{whitteq} on $I_n \cap N(K)$. Accordingly, one can form the {\em adolescent Whittaker categories} $\scc^{I_n, \psi}, n \geqslant 1.$

Due to conjugation by $t^{-n \check{\rho}}$, the $I_n$ no longer form a decreasing series of subgroups. Instead, for $n \leqslant m$, one relates the adolescent Whittaker categories by the functors $$i_{n,m *}: \scc^{I_n, \psi} \xrightarrow{\text{Oblv}} \scc^{I_n \cap I_m, \psi} \xrightarrow{\text{Av}_{*, \psi}} \scc^{I_m, \psi},$$
$$i_{n,m}^!: \scc^{I_m, \psi} \xrightarrow{\text{Oblv}} \scc^{I_n \cap I_m, \psi} \xrightarrow{\text{Av}_{*, \psi}} \scc^{I_n, \psi}.$$Informally, $i_{n,m *}$ averages an element of $\scc$ by $I_m \cap LN / I_n \cap LN$, and $i_{n,m}^!$ averages an element of $\scc$ by $I_n \cap LB^-/ I_m \cap LB^-$. Raskin proves that $i_{n,m}^!$ admit fully faithful left adjoints $i_{n,m !}$. A key result is that this coincides with $i_{n,m*}$ up to a shift: \begin{equation}i_{n,m !} \simeq i_{n,m *}[2(m-n) \Delta]\label{ss},\end{equation}where $\Delta = \dim \text{Ad}_{t^{-\check{\rho}}} L^+N /  L^+N$.

Raskin moreover observes that: $$\scc_{LN, \psi} \simeq \varinjlim_{i_{n,m *}} \scc^{I_n, \psi}, \quad \quad \quad \quad \scc^{LN, \psi} \simeq \varinjlim_{i_{n,m!}} \scc^{I_n, \psi}.$$Here and elsewhere, the above homotopy colimits and limits are taken in the setting of cocomplete dg-categories, unless otherwise specified. By virtue of \eqref{ss}, the maps $\operatorname{id}[-2m \Delta]: \scc^{I_n, \psi} \rightarrow \scc^{I_n, \psi}$ yield the desired isomorphism $\scc^{LN, \psi} \simeq \scc_{LN, \psi}$ of \eqref{whitteq}.

\subsection{Whittaker invariants of dual representations: prounipotent case} 
 Since we are unaware of a reference with proofs for the desired duality and its basic properties in the prounipotent case, though it is well known to experts, we first write this down in some detail.
\label{puc}

\subsubsection{Dualizability}

Suppose $U$ is a quasicompact, prounipotent group scheme. We remind the isomorphism between invariants and coinvariants in more detail. Let $\scc$ be a category strongly acted on by $U$. Then the invariants category $\scc^U$ comes equipped with an adjunction of continuous functors: \begin{equation} \label{adi} \text{Oblv}: \scc^U \leftrightarrows \scc: \text{Av}_*,\end{equation}wherein $\text{Oblv}$ is fully faithful. Similarly, the coinvariants category $\scc_U$ comes equipped with an adjunction of continuous functors: \begin{equation}\label{adc}\text{ins}^L: \scc_U \leftrightarrows \scc: \text{ins},\end{equation}wherein $\text{ins}^L$ is fully faithful. The following assertion may be found in \cite{Beraldo}. 

\begin{theo} The composition $\operatorname{ins} \circ \operatorname{Oblv}: \scc^U \rightarrow \scc_U$ is an equivalence. \label{rlu} \end{theo}

The following compatibility will be useful to us. Throughout the paper, we understand commutative diagrams of dg-categories to commute up to natural equivalence. 

\begin{pro} The equivalence of Theorem \ref{rlu} exchanges the adjunctions \eqref{adi}, \eqref{adc}, i.e. fits into  commutative diagrams: $$\xymatrix{\scc^U \ar[rd]_{\operatorname{Oblv}} \ar[rr]^{\operatorname{ins} \circ \operatorname{Oblv}} & & \scc_U \ar[ld]^{\operatorname{ins}^L} &  & \scc^U \ar[rr]^{\operatorname{ins} \circ \operatorname{Oblv}} & & \scc_U \\   & \scc  & & & & \ar[ul]^{\operatorname{Av_*}} \scc \ar[ur]_{\operatorname{ins}} }$$ \label{swapd}

\end{pro}

\begin{proof} In {\em loc. cit}, an inverse functor $\iota$ to $\operatorname{ins} \circ \operatorname{Oblv}$ is constructed by applying the universal property of $\scc_U$ to $\operatorname{Av_*}: \scc \rightarrow \scc^U$. This yields the right hand triangle, and the other is obtained by passing to left adjoints. \end{proof}

We can now state how dualizability interacts with invariants and coinvariants for prounipotent groups. Recall that  $\operatorname{DGCat}_{cont}$ denotes the $(\infty,2)$-category of cocomplete $\mathbb{C}$-linear dg-categories and continuous $\mathbb{C}$-linear quasi-functors between them, equipped with the Lurie tensor product, cf. \cite{Gaitsge}. Let $\scc$ be a dualizable object of $\operatorname{DGCat}_{cont}$ with dual $\scc^\vee$. If $\scc$ is equipped with a strong action of $U$, then $\scc^\vee$ acquires a strong right action of $U$ by the dualizability of $D(U)$. Applying inversion on the group, one converts this into a strong left action of $U$. 

\begin{pro}\label{dualic} Let $\scc$, $\scc^\vee$ be as in the preceding paragraph. Then $\scc^U$ is again dualizable, with dual $(\scc^\vee)^U$ and pairing: \begin{equation} \label{pairing}  \scc^U \otimes (\scc^\vee)^U  \xrightarrow{\operatorname{Oblv} \otimes \operatorname{Oblv}} \scc \otimes \scc^\vee \rightarrow \operatorname{Vect}.\end{equation}

\end{pro}
\begin{proof} We first show the dualizability of $\scc^U \simeq \scc_U$. For an arbitrary test object $\mathscr{S}$ of $\operatorname{DGCat}_{cont}$, we may equip it with a trivial strong action of $U$, and compute $$\Hom_{\operatorname{DGCat}_{cont}}(\scc_U, \mathscr{S}) \simeq \Hom_{D(U)\operatorname{-mod}}(\scc, \mathscr{S}) \simeq \Hom_{\operatorname{DGCat}_{cont}}(\scc, \mathscr{S})^U \simeq (\scc^\vee \otimes \mathscr{S})^U.$$To show dualizability, it remains to show the natural map $(\scc^\vee)^U \otimes \mathscr{S} \rightarrow (\scc^\vee \otimes \mathscr{S})^U$ is an equivalence, i.e. we would like to pull the $\mathscr{S}$ out of the invariants. Since the latter is a totalization, and in particular an inverse limit, this is not completely formal. However, the analogous assertion for coinvariants is, and so we compute using Theorem \ref{rlu} $$(\scc^\vee \otimes \mathscr{S})^U \simeq (\mathscr{C}^\vee \otimes \mathscr{S})_U \simeq (\mathscr{C}^\vee)_U \otimes \mathscr{S} \simeq (\scc^\vee)^U \otimes \mathscr{S}.$$
By construction, the perfect pairing $\delta: \scc_U \otimes (\scc^\vee)^U \rightarrow \operatorname{Vect}$ fits into a commutative diagram: $$\xymatrix{\scc \otimes (\scc^\vee)^U \ar[r]^{\hspace{2mm} \operatorname{id} \otimes \operatorname{Oblv}} \ar[rd]_{\operatorname{ins} \otimes \operatorname{id}} & \scc \otimes \scc^\vee \ar[r] & \operatorname{Vect.} \\ & \scc_U \otimes (\scc^\vee)^U \ar[ru]_\delta }$$Precomposing the diagram with $\operatorname{ins}^L \otimes \operatorname{id}$, and using the fully-faithfulness of $\operatorname{ins}^L$, we deduce the pairing can be rewritten as $$\scc_U \otimes (\scc^\vee)^U \xrightarrow{ \operatorname{ins}^L \otimes \operatorname{Oblv}} \scc \otimes \scc^\vee \rightarrow \operatorname{Vect}.$$By Theorem \ref{rlu} and Proposition \ref{swapd}, we can replace $\scc_U$ and $\operatorname{ins}^L$ with $\scc^U$ and $\operatorname{Oblv}$, as desired.\end{proof}

Recall that to a continuous functor $F: \scc \rightarrow \mathscr{D}$ between dualizable objects of $\operatorname{DGCat}_{cont}$, one can associate a dual functor $F^\vee: \mathscr{D}^\vee \rightarrow \scc^\vee$. We now determine the effect of duality upon the forgetful and averaging functors, namely that they swap. 

\begin{pro} Under the equivalence $(\scc^U)^\vee \simeq (\scc^\vee)^U$ of Proposition \ref{dualic}, the functors $$\operatorname{Oblv}: \scc^U \leftrightarrows \scc: \operatorname{Av_*}$$are dual to the functors $$\operatorname{Oblv}: (\scc^\vee)^U \leftrightarrows \scc^\vee: \operatorname{Av_*},$$i.e. there are natural isomorphisms $\operatorname{Oblv}^\vee \simeq \operatorname{Av_*}$, $\operatorname{Av_*}^\vee \simeq \operatorname{Oblv}$. 
\label{dualfunc}
\end{pro}

\begin{proof} For a dualizable object $\mathscr{D}$ of $\operatorname{DGCat}_{cont}$ with dual $\mathscr{D}^\vee$, let us write $\langle -, - \rangle_{\mathscr{D}} \in \operatorname{Vect}$ for the pairing $\mathscr{D} \otimes \mathscr{D}^\vee \rightarrow \operatorname{Vect}$. To prove the Proposition, it suffices to provide an equivalence of quasi-functors: $$\langle -, \operatorname{Av}_* - \rangle_{\mathscr{C}^U} \simeq \langle \operatorname{Oblv} -, - \rangle_{\scc}.$$By the definition of the pairing between $\scc^U, (\scc^\vee)^U$, we have $$\langle -, \operatorname{Av}_* - \rangle_{\mathscr{C}^U} \simeq \langle \operatorname{Oblv} -, \operatorname{Oblv} \circ \operatorname{Av}_* - \rangle_{\scc}.$$Recall that $\operatorname{Oblv} \circ \operatorname{Av}$ is given by convolution by $k_U$, the constant sheaf on $U$, cf. \cite{Beraldo}. Writing $\operatorname{inv}: U \rightarrow U$ for inversion, by the definition of the action of $D(U)$ on $\scc^\vee$ we have: $$\langle \operatorname{Oblv} -, \operatorname{Oblv} \circ \operatorname{Av}_* - \rangle_\scc \simeq \langle \operatorname{Oblv} -, k_U \star - \rangle_\scc \simeq \langle \operatorname{inv}_{*, dR} k_U \star \operatorname{Oblv} -, - \rangle_\scc.$$As $\operatorname{inv}_{*, dR} k_U \simeq k_U$, the fully-faithfulness of $\operatorname{Oblv}$ yields: $$\langle k_U \star \operatorname{Oblv} -, - \rangle_\scc \simeq \langle \operatorname{Oblv} -, - \rangle_\scc,$$as desired. Having shown the duality of $\operatorname{Oblv}: \scc^U \rightarrow \scc$ and $\operatorname{Av}_*: \scc^\vee \rightarrow (\scc^\vee)^U$, the other claim either follows from the involutivity of taking dual categories. \end{proof}

Finally, we will need a Whittaker twist of the above. So, suppose $\chi: U \rightarrow \mathbb{G}_a$ is a character. Associated to this is a character representation $\operatorname{Vect}_\chi$. Let us recall the construction. Write $z$ for the standard coordinate on $\mathbb{G}_a$, and $D$ for the algebra of global differential operators on $\mathbb{G}_a$. Then one has the exponential (left) $D$-module: \begin{equation}\label{exp}e^z := D/D(\partial_z - 1) [1].\end{equation}Writing $e^\chi$ for $\chi^! e^z$, then $e^\chi$ is a character $D$-module on $U$, i.e. is equipped with an isomorphism $m^! e^\chi \simeq e^\chi \boxtimes e^\chi,$ satisfying an associativity compatibility, where $m: U \times U \rightarrow U$ is the multiplication map. Accordingly, one has a representation $\operatorname{Vect}_\chi$ of $D(U)$, whose underlying dg-category is $\operatorname{Vect}$, and whose coaction: $$\operatorname{Vect} \xrightarrow{a^!} \operatorname{Vect} \otimes D(U) \simeq D(U) $$sends $V \in \operatorname{Vect}$ to $V \otimes e^\chi$. Note that when $\chi = 0$, $e^\chi$ is the dualizing sheaf $\omega_U$, and this recovers the usual trivial representation of $D(U)$. 

Let us recall that these multiply in the expected in way:

\begin{lemma}\label{dualtw} Let $\chi, \chi'$ be characters of $U$. Then there is a canonical isomorphism of strong representations $\operatorname{Vect}_\chi \otimes \operatorname{Vect}_{\chi'} \simeq \operatorname{Vect}_{\chi + \chi'}.$ In particular, we have a $D(U)$-equivariant duality between $\operatorname{Vect}_{\chi}$ and $ \operatorname{Vect}_{-\chi}.$

\end{lemma}
\begin{proof} This follows from the isomorphism of character $D$-modules: $e^{\chi} \overset{!}{\otimes} e^{\chi'} \simeq e^{\chi + \chi'}$. \end{proof}

For a strong representation $\scc$ of $D(U)$ and an additive character $\chi$, let us write $\scc(\chi)$ for the representation $\scc \otimes \operatorname{Vect}_\chi$. With this, we have by Lemma \ref{dualtw}: $$\scc^{U, \chi} := \operatorname{Hom}_{D(U)\operatorname{-mod}}(\operatorname{Vect}_{\chi}, \scc) \simeq \operatorname{Hom}_{D(U)\operatorname{-mod}}(\operatorname{Vect}, \scc(-\chi) ) = \scc(-\chi)^U.$$In particular, as before there is an adjunction: $$\operatorname{Oblv}_\chi: \scc^{U, \chi} \leftrightarrows \scc: \operatorname{Av}_{\chi, *},$$wherein $\operatorname{Oblv}_\chi$ is fully faithful. Let us now record the interaction of dualizability with twisted invariants. 

\begin{pro} \label{charpairing} Let $\scc$ be a strong representation of $U$, and $\chi: U \rightarrow \mathbb{G}_a$ a character. Then if $\scc$ is dualizable, then the pairing: $$\scc^{U, \chi} \otimes (\scc^\vee)^{U, -\chi} \xrightarrow{\operatorname{Oblv} \otimes \operatorname{Oblv}} \scc \otimes \scc^\vee \rightarrow \operatorname{Vect}$$gives an equivalence $\scc^{U, \chi, \vee} \simeq (\scc^\vee)^{U, - \chi}.$With respect to this duality datum, we have equivalences of functors: $$\operatorname{Oblv}_\chi^\vee \simeq \operatorname{Av}_{-\chi, *} \quad \quad \operatorname{Av}_{\chi, *}^\vee \simeq \operatorname{Oblv}_{-\chi}.  $$

\end{pro}

\begin{proof} We compute $$(\scc^{U, \chi})^\vee \simeq (\scc(-\chi)^U)^\vee \simeq ((\scc(-\chi))^\vee)^U \simeq (\scc^\vee(\chi))^U \simeq (\scc^{\vee})^{U, - \chi},$$where in the middle we used Proposition \ref{dualic}. The remaining assertions can be deduced from applying Propositions \ref{dualic} and \ref{dualfunc} to $\scc(-\chi)$. \end{proof}

\subsubsection{Dualizability and compact objects: generalities}
Let $\mathscr{D}$ be a cocomplete dg-category. Recall an object $d$ of $\mathscr{D}$ is compact if $\Hom(d, -): \mathscr{D} \rightarrow \operatorname{Vect}$ commutes with direct sums. 
Writing $\mathscr{D}^c$ for the subcategory of compact objects, we obtain a map $$\mathscr{D}^{c, op} \rightarrow \operatorname{Hom}(\mathscr{D}, \operatorname{Vect}). $$ Then we have:

\begin{pro} \label{cduals}For $d$ as above, $d^\vee$ is a compact object of $\operatorname{Hom}(\mathscr{D}, \operatorname{Vect})$. Moreover, if $\mathscr{D}$ is dualizable, then using the canonical identification: $$\operatorname{Hom}(\mathscr{D}, \operatorname{Vect}) \simeq \operatorname{Hom}(\operatorname{Vect}, \mathscr{D}^\vee) \simeq \mathscr{D}^\vee,$$the assignment $d \rightarrow d^\vee$ yields an equivalence: \begin{equation}\mathbb{D}: \mathscr{D}^{c, op} \simeq \mathscr{D}^{\vee, c}.\label{dc}\end{equation}
\end{pro}

\begin{proof} Let us write momentarily $\operatorname{Hom}^{naive}_{\operatorname{DGCat}}(\mathscr{D}, \operatorname{Vect})$ for the dg-category of {\em bona fide} dg-functors from $\mathscr{D}$ to $\operatorname{Vect}$ localized at quasi-isomorphisms. Let us similarly write\\ $\operatorname{Hom}^{naive}_{\operatorname{DGCat}_{cont}}(\mathscr{D}, \operatorname{Vect})$ for the full subcategory of continuous dg-functors. If we write $\Hom_{\operatorname{DGCat}}(\mathscr{D}, \operatorname{Vect})$ for the inner Hom between $\mathscr{D}$ and $\operatorname{Vect}$ in $\operatorname{DGCat}$, i.e. the dg-category of quasi-functors, then there is a natural map: \begin{equation} \label{ntor} \operatorname{Hom}^{naive}_{\operatorname{DGCat}}(\mathscr{D}, \operatorname{Vect}) \rightarrow \operatorname{Hom}_{\operatorname{DGCat}}(\mathscr{D}, \operatorname{Vect}).\end{equation}It is known that Equation \eqref{ntor} is an equivalence, cf. \cite{toto}. Recalling that Hom in $\operatorname{DGCat}_{cont}$ is the full subcategory of $\operatorname{Hom}_{\operatorname{DGCat}}(\mathscr{D}, \operatorname{Vect})$ consisting of continuous quasi-functors, it follows that Equation \eqref{ntor} induces an equivalence: \begin{equation} \operatorname{Hom}^{naive}_{\operatorname{DGCat}_{cont}}(\mathscr{D}, \operatorname{Vect}) \simeq \operatorname{Hom}(\mathscr{D}, \operatorname{Vect}).\label{ezhom} \end{equation}The upshot is that there is no distinction between quasi-functors and functors to $\operatorname{Vect}$, i.e. we may work naively. Recall by the 
dg-Yoneda lemma, we have for $d \in \mathscr{D}$ and $\xi \in \operatorname{Hom}^{naive}_{\operatorname{DGCat}}(\mathscr{D}, \operatorname{Vect})$, evaluation on $d$ yields an equivalence: $$\operatorname{Hom}(d^\vee, \xi) \xrightarrow{\sim} \langle d, \xi \rangle,$$where $\langle d, \xi \rangle \in \operatorname{Vect}$ denotes the evaluation of $\xi$ on $d$. It follows that any $d^\vee$ is a compact object of $\operatorname{Hom}_{\operatorname{DGCat}}^{naive}(\mathscr{D}, \operatorname{Vect})$, as sums of functors are computed `point-wise'. In particular, if $d$ is compact, then $d^\vee$ is a compact object of $\operatorname{Hom}(\mathscr{D}, \operatorname{Vect})$. Applying dg-Yoneda therefore yields a fully-faithful embedding $\mathscr{D}^{c, op} \rightarrow \mathscr{D}^{\vee, c}.$ The involutivity of duality provides a reverse map $\mathscr{D}^{\vee, c} \rightarrow \mathscr{D}^{c, op},$ and these are readily seen to be inverse equivalences. 
\end{proof}

\begin{re}This recovers the well known fact that if $\mathscr{D}$ is compactly generated, then $\mathscr{D}$ is dualizable with dual $\operatorname{Ind}(\mathscr{D}^{c, op})$, cf. \cite{gaitsroz}. However, we do not know a reference for \eqref{dc} in the non-compactly-generated case.  \end{re}

\subsubsection{Dualizability and compact objects: prounipotent case} Let us see how compact objects transform under the functors discussed in Section \ref{puc}. In particular, $U$ is still a quasicompact, prounipotent group, $\chi: U \rightarrow \mathbb{G}_a$ is a character, and $\scc$ is a strong representation of $U$. Let us begin with an orienting observation.

\begin{pro} An object $c$ of $\scc^{U, \chi}$ is compact if and only if $\operatorname{Oblv} c$ is. 
\label{cinu}

\end{pro}

\begin{proof} If $c$ is compact, then $\operatorname{Oblv} c$ is compact by the continuity of $\operatorname{Av}_{*, \psi}$. If $\operatorname{Oblv} c$ is compact, then $c$ is by the fully-faithfulness of $\operatorname{Oblv}$. \end{proof}

We will repeatedly use the following general observation to trace what duality does to compact objects.

\begin{pro} Suppose $\mathsf{F}: \mathscr{C} \rightarrow \mathscr{D}$ is a continuous functor between cocomplete dg-categories. If $\mathscr{C}$ and $\mathscr{D}$ are dualizable, and $\mathsf{F}$ admits a continuous right adjoint $\mathsf{G}$, then the following diagram commutes: 
\label{followduals}
$$\xymatrix{\scc^{c,op} \ar[r]^{\mathsf{F}} \ar[d]^{\mathbb{D}} & \mathscr{D}^{c,op} \ar[d]^{\mathbb{D}} \\ \scc^{\vee, c } \ar[r]^{\mathsf{G}^\vee} & \mathscr{D}^{\vee, c}.}$$

\end{pro}

\begin{proof} The continuity of $\mathsf{G}$ implies $\mathsf{F}$ preserves compactness. If $c \in \scc^c$, the asserted commutativity follows from: $$\langle \mathsf{G}^\vee c^\vee, - \rangle_\mathscr{D} \simeq \langle c^\vee, \mathsf{G} - \rangle_\mathscr{C} \simeq \Hom_\scc(c, \mathsf{G} - ) \simeq \Hom_\mathscr{D}(\mathsf{F}c, -) \simeq \langle (\mathsf{F}c)^\vee, - \rangle_\mathscr{D}.$$\end{proof}

Combining Propositions \ref{charpairing}, \ref{followduals}, we obtain 

\begin{cor} For a strong representation $\scc$ of $U$ and a character $\chi: U \rightarrow \mathbb{G}_a$, taking duals commutes with $\operatorname{Oblv}$, i.e. we have a commutative diagram: 

$$\xymatrix{\scc^{U, \chi, c,op} \ar[d]^{\mathbb{D}} \ar[r]^{\operatorname{Oblv}} & \scc^{c,op} \ar[d]^{\mathbb{D}} \\ \scc^{\vee, U, -\chi,c}  \ar[r]^{\operatorname{Oblv}}& \scc^{\vee, c}.}$$

\end{cor}

\subsection{Whittaker invariants of dual representations of loop groups} Having handled the prounipotent case, we can now describe in the interactions between dualizable representations of loop groups, their Whittaker models, and compact objects therein. 
\subsubsection{Dualizability}

Write $D_\kappa(LG)\operatorname{-mod}$ for the $(\infty, 2)$-category of strong level $\kappa$ representations of $LG$. If $\scc$ in $D_\kappa(LG)\operatorname{-mod}$ is a dualizable object of $\text{DGCat}_{cont}$, then $\scc^\vee$ naturally is a right module for $D_\kappa(LG)$, i.e. lies in $D_{-\kappa}(LG)\operatorname{-mod}$. Writing $D_\kappa(LG)\operatorname{-mod}^{\operatorname{dualizable}}$ for the full subcategory of such objects, we obtain a functor $$\mathbb{D}: D_\kappa(LG)\operatorname{-mod}^{\operatorname{dualizable}} \rightarrow D_{-\kappa}(LG)\operatorname{-mod}^{\operatorname{dualizable, op}}$$

Next let us explain why we did not consider twisted differential operators in the discussion for prounipotent groups. Let $\kappa'$ be a level corresponding to a central extension $\widetilde{LG}$ of $LG$, e.g. $\kappa_{Killing}$. Then on any ind-prounipotent subgroup $U$ of $LG$ the restricted central extension $\widetilde{U}$ uniquely splits. Then by ind-prounipotence it is split uniquely, yielding a canonical identification for any $\kappa$: \begin{equation} \label{rll} D_\kappa(U)\operatorname{-mod} \simeq D(U)\operatorname{-mod}.\end{equation}By construction, under these identifications dualizing commutes with restriction from $D_{\pm \kappa}(LG)$ to $D(U)$. 

In particular, the above discussion applies to the subgroups $LN$ and $I_n, n \geqslant 1,$ showing up in the adolescent Whittaker construction. As the Whittaker model involves multiple of these subgroups at once, let us note that for $U' \subseteq U$ ind-prounipotent subgroups of $LG$, by the uniqueness of splittings of $\widetilde{U'}, \widetilde{U}$, the following diagram commutes: $$\xymatrix{D_\kappa(LG)\operatorname{-mod} \ar[r]^{\operatorname{Res}} \ar[d]^{\operatorname{Res}} & D_\kappa(U)\operatorname{-mod} \ar[r]^{\eqref{rl}} & D(U)\operatorname{-mod} \ar[d]^{\operatorname{Res}} \\ D_\kappa(U')\operatorname{-mod} \ar[rr]^{\eqref{rl}} & & D(U')\operatorname{-mod}. }$$ 

We are now ready to state the interaction of duality and Whittaker models. 

\begin{theo}\label{wduals} Let $\scc$ be an object of $D_\kappa(LG)\operatorname{-mod}^{\operatorname{dualizable}}$ with dual representation $\scc^\vee$. Then for a nondegenerate character $\psi: N \rightarrow \mathbb{G}_a$, there is a canonical duality of Whittaker models $$(\scc_{LN, \psi})^\vee {\simeq} (\scc^\vee)_{LN, -\psi} $$such that for any $n \geqslant 1$ the following diagram commutes: \begin{equation}\xymatrix{\scc^{I_n, \psi} \otimes \scc^{\vee, I_n, -\psi} \ar[rrr]^{\operatorname{Oblv} \otimes \operatorname{Oblv} [2n\Delta]} \ar[d]_{\operatorname{ins} \otimes \operatorname{ins}} &&& \scc \otimes \scc^\vee  \ar[d] \\ \scc_{LN, \psi} \otimes \scc^\vee_{LN, -\psi } \ar[rrr] && & \operatorname{Vect}.}\label{wpair} \end{equation}

\end{theo}

\begin{re} To remove the shift by $[2n\Delta]$ in \eqref{wpair} the reader may prefer to write instead $\scc^{LN, \psi, \vee} {\simeq}\scc^\vee_{LN, -\psi}$. We choose to work with coinvariants in view of $\cW_\kappa\operatorname{-mod}$. 

\end{re}

\begin{proof} Let us recall the adolescent Whittaker presentations of the Whittaker model: $$\scc_{LN, \psi} \simeq \varinjlim_{i_{n,m, *}} \scc^{I_n, \psi} \quad \quad \scc^{LN, \psi} \simeq \varprojlim_{i_{n,m}^!} \scc^{I_n, \psi}.$$We will obtain the desired result from the following more general lemma.

\begin{lemma} Consider a diagram of categories $\scc_i, i \geqslant 1,$ in $\operatorname{DGCat}_{cont}$: $$\scc_1 \xrightarrow{F_{2,1}} \scc_2 \xrightarrow{F_{3,2}} \scc_3  \xrightarrow{F_{4,3}} \cdots$$Suppose that each $\scc_i$ is dualizable, and that $F_{i+1, i}$ are fully faithful and admit continuous right adjoints $G_{i,i+1}$. Then there is a canonical perfect pairing: $$(\varinjlim_{F_{n,m}} \scc_n)^\vee {\simeq} \varinjlim_{G_{m,n}^\vee} \scc_n^\vee,$$such that that for each $n \geqslant 1$ the following diagram commutes: \begin{equation}\label{pair}\xymatrix{\scc_n \otimes \scc_n^\vee \ar[rd] \ar[d]_{\operatorname{ins}_n \otimes \operatorname{ins}_n} \\ \varinjlim \scc_n \otimes \varinjlim \scc_n^\vee \ar[r] & \operatorname{Vect},}\end{equation}where $\operatorname{ins_n}$ is the tautological map into the colimit. \end{lemma}

\begin{proof} This follows from \cite{Gaitsge} Lemma 2.2.2, where to obtain the diagram \eqref{pair} one should notice the simplification of Lemma 1.3.6 of {\em loc. cit} in the case where the $F_{n,m}$ are fully faithful.  \end{proof}

To apply the lemma, we recall that $i_{n,m, *}$ was the composition: $$i_{n,m, *}: \scc^{I_n, \psi} \xrightarrow{\operatorname{Oblv}} \scc^{I_n \cap I_m, \psi} \xrightarrow{\operatorname{Av}_{*, \psi}} \scc^{I_m, \psi}.$$The latter relative averaging is by definition the right adjoint to $\operatorname{Oblv}: \scc^{I_m, \psi} \rightarrow \scc^{I_m \cap I_n, \psi}.$We claim that it may be calculated as the composition: $\scc^{I_n \cap I_m, \psi} \xrightarrow{\operatorname{Oblv}} \scc \xrightarrow{\operatorname{Av}_{*, \psi}} \scc^{I_m, \psi}.$ Indeed, we note that the fully-faithfulness of the forgetful functors $\scc^{I_n, \psi} \rightarrow \scc, \scc^{I_m \cap I_n, \psi} \rightarrow \scc$ imply the fully-faithfulness of $\scc^{I_n, \psi} \xrightarrow{\operatorname{Oblv}} \scc^{I_n \cap I_m, \psi}$, from which we calculate for $c_{n,m} \in \scc^{I_n \cap I_m, \psi}$, $c_m \in \scc^{I_m, \psi}$: $$\Hom_{\scc^{I_n \cap I_m, \psi}} ( \operatorname{Oblv} c_m , c_{n,m}) \simeq \Hom_{\scc}(\operatorname{Oblv} c_m, \operatorname{Oblv} c_{n,m}) \simeq \Hom_{\scc^{I_m, \psi}}(c_m, \operatorname{Av}_{*, \psi} \circ \operatorname{Oblv} c_{n,m}). $$By a similar argument for $i_{n,m}^!$, we have rewritten the functors as: $$i_{n,m, *}: \scc^{I_n, \psi}\xrightarrow{\operatorname{Oblv}} \scc \xrightarrow{\operatorname{Av}_{*, \psi}} \scc^{I_m, \psi}, \quad \quad i_{n,m}^!: \scc^{\vee, I_m, -\psi} \xrightarrow{\operatorname{Oblv}} \scc^\vee \xrightarrow{\operatorname{Av}_{*, -\psi}} \scc^{\vee, I_n, -\psi}.$$By Proposition \ref{swapd}, it follows that $i_{n,m, *}$ and $i_{n,m}^!$ are dual functors.  

We are ready to apply the lemma. The $F_{n,m}$ are $i_{n,m, *}$, the $G_{m,n}$ are $i^!_{m,n}[2(m-n)\Delta]$, and hence by the preceding paragraph the $G^\vee_{m,n}$ are $i_{n,m, *}[2(m-n)\Delta] \simeq i_{n,m,!}$. To conclude, we use that $$\varinjlim_{i_{n,m, *}} \scc^{\vee, I_n, -\psi} \simeq \varinjlim_{i_{n,m,*}[2(m-n)\Delta]} \scc^{\vee, I_n, -\psi},$$where we take the colimit of the functors $\operatorname{id}[2n \Delta]: \scc^{\vee, I_n, -\psi} \rightarrow \scc^{\vee, I_n, -\psi}.$\end{proof}

\subsubsection{Dualizability and compact objects: Whittaker models}

We first explain the relation between compact objects of the adolescent and full Whittaker models by proving the following general proposition, which roughly says taking compact objects commutes with taking unions in $\operatorname{DGCat}_{cont}$. 

\begin{pro} \label{cunions}Let $I$ be a filtered $(\infty, 1)$-category and $\scc_i, i \in I,$ an $I$ system of objects in $\operatorname{DGCat}_{cont}$. Suppose for each 1-morphism $\alpha: i \rightarrow j$ in $I$ the corresponding functor $F_\alpha: \scc_i \rightarrow \scc_j$ is fully faithful and admits a continuous right adjoint $G_\alpha$. Then for each $i_0 \in I$, the tautological functor $\scc_{i_0} \rightarrow \varinjlim_{i \in I} \scc_i$ preserves compactness, and this induces an equivalence: \begin{equation} \label{cl} \varinjlim_{i \in I} \scc_i^c \xrightarrow{\sim} (\varinjlim_{i \in I} \scc_i)^c\end{equation}

\end{pro}
\begin{re} To be clear, in the left hand side of Equation \eqref{cl} the colimit is taken in $\operatorname{DGCat}$, i.e. is a colimit of non-cocomplete dg-categories, and on the right hand side one is taking compact objects from a colimit in $\operatorname{DGCat}_{cont}$. 

\end{re}

\begin{proof} For fixed $\iota \in I$, we claim the insertion $\operatorname{ins}_\iota : \scc_\iota \rightarrow \varinjlim \scc_j$ is fully faithful and admits a continuous right adjoint. To see this, we will show that under the isomorphism $$ \varinjlim_{F_\alpha} \scc_i  \simeq \varprojlim_{G_\alpha} \scc_i,$$the desired right adjoint is given simply by the evaluation $\operatorname{ev}_\iota: \varprojlim \scc_i \rightarrow \scc_\iota$. Indeed, for an object of $\varprojlim \scc_i$, which concretely is realized as a homotopy coherent system $\varprojlim c_i$ of objects of $\scc_i, i \in I,$ and an object $c \in \scc_\iota$ we have: $$\operatorname{Hom}( \operatorname{ins}_\iota c, \varprojlim c_i) \simeq \varprojlim \operatorname{Hom}(\operatorname{ev}_i \circ \operatorname{ins}_\iota  c, c_i )$$By filteredness of $I$, we may without loss of generality run the above homotopy limit of complexes over indices $i$ admitting a map $\beta: \iota \rightarrow i$. Moreover, by the fully-faithfulness of the $F_\alpha$, for such $i$ and $\beta$, we have $\operatorname{ev}_i \circ \operatorname{ins}_\iota \simeq F_\beta c$, and hence: $$\varprojlim \operatorname{Hom}(\operatorname{ev}_i \circ \operatorname{ins}_\iota c, c_i) \simeq \varprojlim \operatorname{Hom}( F_\beta c, c_i) \simeq \varprojlim \operatorname{Hom}(c, G_\beta c_i) \simeq \operatorname{Hom}(c, \operatorname{ev}_\iota \varprojlim c_i).$$Applying the above discussion of pairs $i$ and $\beta$ to $\iota$ and $\operatorname{id}: \iota \rightarrow \iota$ shows that $\operatorname{ev}_\iota$ is fully faithful. 

Having shown that $\operatorname{ins}_\iota$ admits a continuous right adjoint, it follows that it preserves compactness, and hence we obtain the map \eqref{cl}. Recall that filtered colimits in $\operatorname{DGCat}$ can be computed very naively, with every object inserted from a step in the colimit and morphisms the colimit of stepwise morphisms, cf. \cite{Roz}. By the fully-faithfulness of the $F_\alpha$ and $\operatorname{ins}_\iota$ it follows that $\eqref{cl}$ is fully-faithful. To see that it is essentially surjective on homotopy categories, observe that the composites $\operatorname{ins}_\iota \circ \operatorname{ev}_\iota$, $\iota \in I$, assemble into a colimit as one varies $\iota$, in a manner compatible with the counits $\operatorname{ins}_\iota \circ \operatorname{ev}_\iota \rightarrow \operatorname{id}$. Moreover, the resulting map $\varinjlim_{\iota \in I} \operatorname{ins}_\iota \circ \operatorname{ev}_\iota \rightarrow \operatorname{id}$ is a natural isomorphism. Applying this to a compact object $c$ of $\varinjlim \scc_i$, by compactness the inverse map $c \rightarrow \operatorname{ins}_\iota \circ \operatorname{ev}_\iota c$ factors through $\operatorname{ins}_i \circ \operatorname{ev}_i c$ for some $i \in I$. Thus $c$ is a direct summand of $\operatorname{ins}_i \circ \operatorname{ev}_i c$, and by the fully-faithfulness of $\operatorname{ins}_i$ moreover of the form $\operatorname{ins}_i \tilde{c}$. Again by the fully-faithfulness of $\operatorname{ins}_i$, it follows that $\tilde{c}$ is compact, as desired. \end{proof}

With these preparations, we can state how duality acts on compact objects in Whittaker models. 

\begin{theo} \label{compactsinwhs}Let $\scc$, $\scc^\vee$, $\psi$, and $(\scc_{LN, \psi})^\vee {\simeq} (\scc^{\vee})_{LN, -\psi}$ be as in Theorem \ref{wduals}. On compact objects, Theorem \ref{wduals}, combined with  Propositions \ref{cduals}, \ref{cunions} yield:$$\varinjlim_{i_{n,m *}} \scc^{I_n, \psi, c, op} \simeq \scc_{LN, \psi}^{c, op} {\simeq} \scc^{\vee,c}_{LN, -\psi} \simeq \varinjlim_{i_{n,m *}} \scc^{\vee, I_n, -\psi, c}.$$If we write $\mathbb{D}_n: \scc^{I_n, \psi, c, op} \rightarrow \scc^{\vee, I_n, -\psi,c}$ for the equivalence induced by Proposition \ref{charpairing}, then the composite equivalence $$\mathbb{D}_\infty : \varinjlim_{i_{n,m *}} \scc^{I_n, \psi, c, op} {\simeq} \varinjlim_{i_{n,m *}} \scc^{\vee, I_n, -\psi, c}$$is given by the colimit of $[-2n 
\Delta] \circ \mathbb{D}_n: \scc^{I_n, \psi, c, op} \simeq \scc^{\vee, I_n, -\psi, c}.$\end{theo}
\begin{proof}This follows from the construction of the pairings in Theorem \ref{wduals}, Proposition \ref{cduals}, and Proposition \ref{cunions}. Nonetheless, let us provide some detail for the convenience of the reader. Let us write ${}_n\langle -, - \rangle$ for the pairing between $\scc^{I_n, \psi}$ and $\scc^{\vee, I_n, -\psi}$, and ${}_\infty \langle -, - \rangle$ for the pairing between $\scc_{LN, \psi}$ and $\scc^\vee_{LN, -\psi}$. Then by construction, for $\xi_n \in \scc^{\vee, I_n, - \psi}$ and $c_n \in \scc^{I_n , \psi}$, we have: $${}_\infty\langle \operatorname{ins} \xi_n, \operatorname{ins} c_n \rangle \simeq {}_n\langle \xi_n, c_n \rangle [2n \Delta].$$To get the correct shifts, for $c_n$ a compact object of $\scc^{I_n, \psi}$, and $d_n$ an arbitrary object of $\scc^{I_n, \psi}$ we compute: $${}_\infty \langle (\operatorname{ins} c_n)^\vee, \operatorname{ins} d_n \rangle \simeq \operatorname{Hom}(\operatorname{ins} c_n, \operatorname{ins} d_n) \simeq \operatorname{Hom}( c_n, d_n) $$$$\simeq {}_n \langle c_n^\vee, d_n \rangle \simeq {}_n \langle c_n^\vee[-2n \Delta], d_n \rangle[2n \Delta] \simeq  {}_\infty \langle \operatorname{ins}( c_n^\vee[-2n \Delta]), d_n \rangle.$$\end{proof}

\section{Categorical Feigin--Fuchs duality and semi-infinite cohomology for $\cW$-algebras} 
This section is structured as follows. We first prove that semi-infinite cohomology gives a duality for Tate Lie algebras. We then reverse this logic for $\cW$, i.e. we apply the above duality in the case of Kac--Moody and pass to Whittaker models to construct semi-infinite cohomology for $\cW$. We then sketch an example calculation of this semi-infinite cohomology anticipated by I. Frenkel and Styrkas.

\subsection{Feigin--Fuchs duality for Tate Lie algebras}
The goal of this subsection is Theorem \ref{tld}, which says that semi-infinite cohomology is a perfect pairing between Tate Lie algebra representations at complementary levels. For Kac--Moody this is a theorem of Gaitsgory--Arkhipov \cite{ag}. The proof of {\em loc. cit.} uses explicit kernel bimodules. As we explain, all that one needs is the Shapiro lemma. 

\subsubsection{Reminders on Tate Lie algebras} Some references for Tate Lie algebras are Sections 2.7.7, 3.8.17 of \cite{bd} and Appendix D of \cite{pos}.

Let $L$ be a Tate Lie algebra, and write $L\operatorname{-mod}^\heartsuit$ for the abelian category of smooth $L$ representations. By a central extension $L_k$ of $L$, we mean a Tate Lie algebra $L_k$ fitted into an exact sequence: $$0 \rightarrow \mathbb{C} \mathbf{1} \rightarrow L_k \rightarrow L \rightarrow 0,$$where $\mathbf{1} \in L_k$ is central. For a central extension $L_k$, write $L_k\operatorname{-mod}^\heartsuit$ for the full subcategory of its abelian category of discrete modules on which $\mathbf{1}$ acts as the identity. For any Tate Lie algebra $L$, one has its canonical Tate extension $L_{\operatorname{Tate}}$. For any compact open subalgebra $L_0 \subset L$, one has a canonical section $s_0: L_0 \rightarrow L_{\operatorname{Tate}}$;  these need not be compatible under inclusions of compact open subalgebras. 

For any compact open subalgebra $L_0 \subset L,$ one has the functor of semi-infinite cohomology: $$C^{\frac{\infty}{2} + *}(L_{-\operatorname{Tate}}, L_0, -): L_{-\operatorname{Tate}}\operatorname{-mod}^+ \rightarrow \operatorname{Vect}.$$For a concrete presentation, ostensibly written in the setting of a $\mathbb{Z}$-graded Lie algebra, see Section 2 of \cite{ffsi}. Let us remind how this mildly depends on $L_0$. Recall that for any two compact open subspaces $C_0, C_1$ of a Tate vector space $V$, one can form their relative determinant $\det(C_0, C_1)$. If we choose a compact open subspace $C_2$ containing $C_0$ and $C_1$, then there is a canonical isomorphism: $$\det(C_0, C_1) \simeq \det( C_2/C_0 ) \otimes \det( C_2/C_1)^\vee.$$For two compact open subalgebras $L_0, L_1$, one has isomorphisms $$C^{\frac{\infty}{2} + *}(L_{-\operatorname{Tate}}, L_0, -) \simeq C^{\frac{\infty}{2} + *}(L_{-\operatorname{Tate}}, L_1, -) \otimes \det(L_0, L_1),$$where $\det(L_0, L_1)$ is graded by viewing $L$ as a graded vector space of degree -1. 

The canonical section $s_0: L_0 \rightarrow L_{\operatorname{Tate}}$ induces a section $L_0 \rightarrow L_{-\operatorname{Tate}}$, hence one can induce representations of $L_0$ to $L_{-\operatorname{Tate}}$. We will need the following standard:

\begin{pro} (Shapiro lemma) For any $M \in L_0\operatorname{-mod}^\heartsuit$, there is a canonical equivalence \begin{equation}
C^{\frac{\infty}{2} + *}(L_{-\operatorname{Tate}}, L_0, \operatorname{ind}_{L_0}^{L_{-\operatorname{Tate}}} M ) \simeq C^*(L_0, M).    
\end{equation} \label{shap}\end{pro}
\begin{proof}Since we are unaware of a reference containing a proof, we sketch one. Write $L^*$ for the continuous dual of $L$, and $Cl(L[1] \oplus L^*[-1])$ for the graded topological Clifford algebra generated by $L[1] \oplus L^*[-1]$ with its tautological pairing. This has a unique, up to tensoring by a graded line, simple discrete graded module. For a choice of open Lagrangian $U \subset L[1] \oplus L^*[-1]$, one can realize this module as $\operatorname{ind}_{\operatorname{Sym} U}^{Cl(L[1] \oplus L^*[-1)]} \mathbb{C}$, where $\mathbb{C}$ is the unique simple representation of the topological exterior algebra $\operatorname{Sym} U$. Let us write $Sp$ for its realization associated to the Lagrangian $L_0 \oplus L_0^\perp$. Writing $Sp^{L_0^\perp}$ for the subspace killed by $L_0^\perp$, we have: \begin{equation}Sp^{L_0^\perp} = \operatorname{Sym} (L^*[-1]) \mathbb{C} = \operatorname{Sym}(L^*/L_0^\perp[-1]) \mathbb{C} = \operatorname{Sym}(L_0^*[-1]) \mathbb{C}.\label{chains}\end{equation}

Recall the underlying vector space of $C^{\frac{\infty}{2} + *}(L_{-\operatorname{Tate}}, L_0, -)$ is given by $Sp \otimes -$. Equation \eqref{chains} identifies $Sp^{L_0^\perp} \otimes M$ with $C^*(L_0, M)$, and this is compatible with differentials. 

It remains to show this inclusion is a quasi-isomorphism. To do so, filter $\operatorname{ind} M := \operatorname{ind}_{L_0}^{L_{-\operatorname{Tate}}} M$ by $$F^i \operatorname{ind} M := (F^i U(L_{-\operatorname{Tate}})) M, \quad \quad i \geqslant 0,$$where $U(L_{-\operatorname{Tate}})$ denotes the universal enveloping algebra, and  $F^iU(L_{-\operatorname{Tate}})$ is the $i^{th}$ step in its PBW filtration. Writing $Sp^j, j \in \mathbb{Z}$, for the $j^{th}$ graded component of $Sp$, filter $Sp \otimes \operatorname{ind} M$ by $$F^i Sp \otimes \operatorname{ind }M := \bigoplus_{-j + k = i} Sp^j \otimes F^k \operatorname{ind} M.$$This is a filtration by subcomplexes. In the associated spectral sequence for its cohomology, on $E_1$ one encounters a family of Koszul complexes for $\operatorname{Sym} L/L_0$, and the $E_2$ page has the desired form. 
\end{proof}

\subsubsection{Renormalized derived categories}

Recall that $L$ is a topological Lie algebra with central extension $L_k$. Since the projection $L_k \rightarrow L$ is open, one has:

\begin{lemma} Write $\mathbf{C}_L$ for the category whose objects are compact open subalgebras $C$ of $L$, and whose morphisms are inclusions. Write $\mathbf{C}_{L_k}$ for the category whose objects are compact open subalgebras $C_k$ of $L_k$ containing $\mathbf{1}$, and whose morphisms are inclusions. Then pullback defines an equivalence $\mathbf{C}_L \simeq \mathbf{C}_{L_k}$.\label{cos} \end{lemma}

Write $L_k\operatorname{-mod}^n$ ($n$ for naive) for the usual unbounded derived category of $L_k\operatorname{-mod}^\heartsuit$. Within it, consider the objects of the form $\operatorname{ind}_{C_k}^{L_k} V$, where $C_k \in \mathbf{C}_{L_k}$ and $V \in C_k \operatorname{-mod}^\heartsuit$ is a finite dimensional module. Consider the pretriangulated envelope $C$ of these objects within $L_k\operatorname{-mod}^n$,\footnote{We could equivalently work with $V \in C_k\operatorname{-mod}^n$ with finitely many nonzero cohomology groups, each smooth and finitely generated, or just $\operatorname{ind}_{C_k}^{L_k} \mathbb{C}$, the trivial representations.} and set $L_k\operatorname{-mod} := \operatorname{Ind}(C)$. We have a tautological functor: $$\Psi: \operatorname{Ind}(C) \rightarrow L_k\operatorname{-mod}.$$

\begin{pro} \label{basl}$\Psi$ admits a (typically discontinuous) fully faithful right adjoint $$\Psi: L_k\operatorname{-mod} \leftrightarrows L_k\operatorname{-mod}^n: \Phi.$$Moreover, there is a unique $t$-structure on $L_k\operatorname{-mod}$ with $L_k\operatorname{-mod}^{\geqslant 0} = \Phi  L_k\operatorname{-mod}^{n, \geqslant 0}$. In particular, $\Phi$ and $\Psi$ define inverse equivalences between the bounded below subcategories. 
\end{pro}

\begin{proof}Let us explain the fully faithfulness of $\Phi$. Since $L_k\operatorname{-mod}^\heartsuit$ is a Grothendieck abelian category, it follows that $L_k\operatorname{-mod}^n$ has enough homotopy injective complexes, cf. \cite{ser}. Writing $\operatorname{Oblv}: L_k\operatorname{-mod}^n \rightarrow \operatorname{Vect}$ for the forgetful functor, and letting $C_k$ run over the objects of $\mathbf{C}_{L_k}$, cf. Lemma \ref{cos},  it follows that $$\varinjlim \operatorname{Hom}( \operatorname{ind}_{C_k}^{L_k} \mathbb{C}, -) \simeq \operatorname{Oblv}(-).$$The remainder of the argument for fully faithfulness, and those for the $t$-structures, are now identical to those in Section 22 of \cite{fg}. \end{proof}

\subsubsection{Semi-infinite cohomology and duality}

For a central extension $L_k$, tensor product of representations gives a map $$L_k\operatorname{-mod} \otimes L_{-k - \operatorname{Tate}} \rightarrow L_{-\operatorname{Tate}}\operatorname{-mod}.$$Namely, one defines the pairing on our compact generators to be tensor product, and then ind extends. Similarly, recalling that $L_0$ denotes a compact open subalgebra of $L$, one obtains a map $$C^{\frac{\infty}{2} + *}(L_{-\operatorname{Tate}}, L_0, -): L_{-\operatorname{Tate}}\operatorname{-mod}  \rightarrow \operatorname{Vect}.$$We are now ready to prove the main result of this section. 
\begin{theo}The composite pairing \begin{equation}C^{\frac{\infty}{2} + *}(L_{-\operatorname{Tate}}, L_0, - \otimes - ): \operatorname{}L_k\operatorname{-mod} \otimes L_{-k-\operatorname{Tate}} \rightarrow \operatorname{Vect}\label{pinf}\end{equation}is perfect, i.e. is a duality datum in $\operatorname{DGCat}_{cont}$. 
\label{tld}
\end{theo}

\begin{proof} Let us write $C := L_k\operatorname{-mod}, D := L_{-k-\operatorname{Tate}}\operatorname{-mod}$, and $\langle -, - \rangle$ for the pairing \eqref{pinf}. Since both are compactly generated, they are dualizable, with duals and pairings $$\langle - , - \rangle_C: C \otimes C^\vee \rightarrow \operatorname{Vect}, \quad \quad \langle-, - \rangle_D: D \otimes D^\vee \rightarrow \operatorname{Vect}.$$
The pairing \eqref{pinf} yields maps $\phi: C \rightarrow D^\vee, 
\gamma: D^\vee \rightarrow C$ satisfying  $$\langle -, \gamma - \rangle_C \simeq \langle -, - \rangle \simeq \langle \phi - , - \rangle_D.$$ We first observe that these preserve compactness. Indeed, fix a compact open $K \subset L$, pick a finite dimensional $V \in K_{k}\operatorname{-mod}^\heartsuit$, and calculate: $$\langle \phi \operatorname{ind}_{K_k}^{L_k} V, - \rangle_D \simeq \csi( \lt, L_0, \operatorname{ind}_{K_k}^{L_k} (V)\otimes - ) \simeq \csi(\lt, L_0, \operatorname{ind}_{K_{-\operatorname{Tate}}}^{\lt}( V \otimes - )) $$$$
\simeq \csi( \lt, K, \operatorname{ind}_{K_{-\operatorname{Tate}}}^{\lt}( V \otimes - )) \otimes \det(L_0, K) \simeq C^*( K, V \otimes - ) \otimes \det(L_0, K)$$$$\simeq \operatorname{Hom}_{K_{-k - \operatorname{Tate}}\operatorname{-mod}}( V^\vee \otimes \det(K, L_0), -) \simeq \operatorname{Hom}_{L_{-k-\operatorname{Tate}}\operatorname{-mod}}(\operatorname{ind}_{K_{-k-\operatorname{Tate}}}^{L_{-k - \operatorname{Tate}}} (V^\vee \otimes \det(K, L_0)), -).$$Note that in the above calculation, if $-$ is compact, we are using the `naive' definitions of the above functors. Since a general object $-$ of the renormalized derived category is a colimit of compact ones, the same calculation applies due to the continuity of renormalized functors appearing.

I.e., writing $\mathbb{D}$ for the equivalence $\mathbb{D}: D^{c, op} \simeq D^{\vee, c}$, cf. Proposition \ref{cduals}, we have produced an isomorphism \begin{equation}\label{is1}\phi \operatorname{ind}_{K_k}^{L_k} V \simeq \mathbb{D} \operatorname{ind}_{K_{-k-\operatorname{Tate}}}^{L_{-k - \operatorname{Tate}}} (V^\vee \otimes \det(K, L_0)).\end{equation}Interchanging the roles of $C$ and $D$, for any finite dimensional $W \in K_{-k -\operatorname{Tate}}\operatorname{-mod}^\heartsuit$, we obtain an isomorphism \begin{equation}\gamma \operatorname{ind}_{K_{-k-\operatorname{Tate}}}^{L_{-k -\operatorname{Tate}}} W \simeq \mathbb{D} \operatorname{ind}_{K_k}^{L_k} (W^\vee \otimes \det(K, L_0)).\label{is2}\end{equation} 

Having shown that $\phi$ preserves compactness, we may apply the following general lemma.

\begin{lemma} Suppose $C,D$ are dualizable objects of $\operatorname{DGCat}_{cont}$ equipped with a pairing $C \otimes D \rightarrow \operatorname{Vect}.$ Let $\phi, \gamma$ be as above. If $\phi$ preserves compactness, and $C$ is compactly generated, then $\gamma$ admits a left adjoint $\gamma^L$.

\end{lemma}
\begin{proof} Since $\phi$ preserves compactness, we obtain a map $\phi: C^c \rightarrow D^{\vee, c} \simeq D^{c, op}$, and hence its opposite $\mathbb{D} \phi \mathbb{D}: C^{c, op} \rightarrow D^c$. We will show that its ind-extension $\mathbb{D} \phi \mathbb{D}: C^\vee \rightarrow D$ provides the desired left adjoint. If we pick an object $\xi \in C^{\vee, c}$, this follows from $$\operatorname{Hom}_D( \mathbb{D} \phi \mathbb{D} \xi, - ) \simeq \langle \phi \mathbb{D} \xi, - \rangle_D \simeq \langle \mathbb{D} \xi, - \rangle \simeq \langle \mathbb{D} \xi, \gamma - \rangle_C \simeq \operatorname{Hom}_{C^\vee}( \xi, \gamma(-) ).$$\end{proof}
To prove the theorem, by the symmetry between $C$ and $D$ it suffices to show that $\gamma$ is an equivalence. Since $\gamma$ is continuous, preserves compactness, and its essential image contains a set of compact generators of $C^\vee$, it remains to show that $\gamma$ is fully faithful on $D^c$. To do so, we will argue that the natural transformation $\gamma^L \gamma \rightarrow \operatorname{id}_{D^c}$ is an equivalence. Let us introduce some notation. Set $c := \operatorname{ind}_{K_k}^{L_k} V$ and  $\check{c} := \mathbb{D} c$. Set $d := \operatorname{ind}_{K_{-k - \operatorname{Tate}}}^{L_{-k - \operatorname{Tate}}} (V^\vee \otimes \det(K, L_0))$. Consider the diagram: $$\xymatrix{\operatorname{Hom}(\gamma^L \check{c}, d) \ar@{-}[r]^{\hspace{.75cm}\sim} \ar@{-}[d]_{\eqref{is1}} & \langle c, d \rangle \ar@{-}[r]^{\hspace{-.75cm} \sim} & \operatorname{Hom}(\check{c}, \gamma d) \ar@{-}[d]^{\eqref{is2}} \\ \operatorname{Hom}(\gamma^L \check{c}, \gamma^L \check{c}) & & \operatorname{Hom}(\gamma d, \gamma d). }$$To see that $\gamma^L \gamma d \rightarrow d$ is an isomorphism, it remains to observe that the images of $\operatorname{id}_{\gamma^L \check{c}}$ and $\operatorname{id}_{\gamma d}$ coincide in $\langle c, d\rangle$. Explicitly, if we write $$\langle c, d\rangle \simeq \csi(\lt, K, \operatorname{ind}_{K_k}^{L_k} V \otimes  \operatorname{ind}_{K_{-k - \operatorname{Tate}}}^{L_{-k-\operatorname{Tate}}} V^\vee ),$$then both correspond to the image of $\id_V$ under the composition: $$\operatorname{End}(V) \simeq V \otimes V^\vee \rightarrow C^*(K, V \otimes V^\vee) \rightarrow \langle c, d \rangle.$$ \end{proof}

The following equivariance property of the pairing will be important to us. It is known to experts for Kac--Moody. Though we do not know of a published proof, it will appear in forthcoming work of Raskin \cite{mys}, see also \cite{rsi}. 

\begin{theo} Suppose $L$ is the Tate Lie algebra of an  group ind-scheme $\mathbf{H}$, and suppose that $\mathbf{H}$ contains a compact open subgroup $\mathbf{K}$ such that $\mathbf{H}/\mathbf{K}$ is ind-proper. Then for any compact open subalgebra $L_0$ and central extension $L_k$ of $L$, the pairing $$C^{\frac{\infty}{2} + *}(L_{-\operatorname{Tate}}, L_0, - \otimes - ): L_k\operatorname{-mod} \otimes L_{-k-\operatorname{Tate}}\operatorname{-mod} \rightarrow \operatorname{Vect}$$carries a canonical $\mathbf{H}$ equivariant structure.   \end{theo}

\subsection{Feigin--Fuchs duality for $\cW$} We would like to identify representations of $\cW$ at complementary levels $\kappa_c \pm \kappa$ as dual categories. Let us do so now:

\begin{theo}\label{ffd}For any $\kappa$ there is a canonical duality in $\operatorname{DGCat}_{cont}$ $$\cw_\kappa\operatorname{-mod}^\vee  \simeq \cw_{-\kappa + 2 \kappa_c}\operatorname{-mod}.$$

\begin{proof}As in the previous subsection, the functor of semi-infinite cohomology gives an $LG$ equivariant perfect pairing $$C^{ \frac{\infty}{2} + *}(\hat{\fg}_{2\kappa_c}, L^+ \fg, - \otimes - ): \gk\operatorname{-mod} \otimes \hat{\fg}_{-\kappa + 2 \kappa_c} \rightarrow \operatorname{Vect}.$$Let us write $\psi: L\fn  \rightarrow \mathbb{C}$ for the character of $L\fn$ induced by the group homomorphism of Equation \eqref{psi}, and $\mathbb{C}_\psi$ for the corresponding one dimensional representation. For the time being, let us write $\cW_\kappa^\psi := H^{ \frac{\infty}{2}}(L\fn, L^+ \fn, \mathbb{V}_\kappa \otimes \mathbb{C}_\psi).$ As proved by Raskin \cite{r}, the functor $$C^{\frac{\infty}{2} + *}( L\fn, L^+ \fn, - \otimes \mathbb{C}_\psi): \gk\operatorname{-mod} \rightarrow \cW^\psi_\kappa\operatorname{-mod}$$induces an `affine Skryabin' isomorphism \begin{equation}\gk\operatorname{-mod}_{LN, \psi} \simeq \cW_\kappa^\psi\operatorname{-mod}.\label{affs}\end{equation}Combining these with Theorem \ref{wduals}, we obtain: \begin{equation}\label{bigdeal}(\cW_\kappa^{\psi}\operatorname{-mod})^\vee \simeq (\gk\operatorname{-mod}_{LN, \psi})^\vee  \simeq \hat{\fg}_{-\kappa + 2\kappa_c}\operatorname{-mod}_{LN, -\psi} \simeq \cW_{-\kappa + 2\kappa_c}^{-\psi}\operatorname{-mod}.\end{equation}
It remains to explain why the difference between $\pm \psi$ is inessential. Namely, writing $H$ for the adjoint torus corresponding to the finite Cartan $\mathfrak{h} \subset \fg$, its conjugation action on $N$ induces a simply transitive action on the nondegenerate $\psi$. This gives an simply transitive action on the corresponding Whittaker coinvariants:

\begin{lemma}\label{lchp}For a character $\psi$ of $N$, and $h \in H(\mathbb{C})$, consider the conjugated character $h \psi := \psi \circ \operatorname{Ad}_{h^{-1}}.$ Write $\delta_h \in D(H)^\heartsuit$ for the delta D-module supported on $h$. Then for any category $\scc$ with a strong action of $H \ltimes LN$, convolution with $\delta_h$ yields  isomorphisms: \begin{equation}\delta_h \star -: \scc^{LN, \psi} \simeq \scc^{LN, h \psi}, \quad \quad \quad \quad  \delta_h \star - : \scc_{LN, \psi} \simeq \scc_{LN, h \psi}.\label{chp}\end{equation}
\end{lemma}
\begin{proof}Let $U$ be a quasicompact prounipotent subgroup of $LN$ stable under the action of $H$, and by abuse of notation continue to write $\psi, h\psi$ for the restriction of these characters to $U$. We first show the equivalence: \begin{equation}\label{chch}\delta_h \star -: \scc^{U, \psi} \rightarrow \scc^{U, h \psi}.\end{equation}To see this, note that for $\psi: U \rightarrow \mathbb{G}_a$, by smoothness $\psi_{*}: D^*(U) \rightarrow D(\mathbb{G}_a)$ admits a left adjoint $\psi^*$. Write $\underline{e}^{-\psi}$ for the `twisted constant sheaf' $\psi^* e^{-z}[-2]$, cf. Equation \eqref{exp}. Then by the fully-faithfulness of $\operatorname{Oblv}: \scc^{U, \psi} \rightarrow \scc$, $\scc^{U, \psi}$ may be recovered as the full subcategory of $\scc$ consisting of the essential image of $\underline{e}^{-\psi} \star - $. Since $\delta_h \star \delta_{h^{-1}} = \delta_e$, it is equivalently the essential image of $\underline{e}^{-\psi} \star \delta_{h^{-1}} \star - $. Writing $\operatorname{Ad}_h: LN \rightarrow LN$ for conjugation by $h$, we finish by computing: \begin{equation}\label{exps}\delta_h \star \underline{e}^{-\psi} \star \delta_{h^{-1}} \simeq \operatorname{Ad}_{h, *} \underline{e}^{-\psi} \simeq (\psi \circ \operatorname{Ad}_{h^{-1}})^* e^{-z}[-2] \simeq \underline{e}^{-h \psi}.\end{equation}
Write $LN$ as a union of compact open subgroups $U_j, j 
\geqslant 1,$ stable under the action of $H$. Recall that: $$\scc^{LN, \psi} \simeq \varprojlim_{\operatorname{Oblv}} \scc^{U_j, \psi} \quad \quad \scc_{LN, \psi} \simeq \varinjlim_{\operatorname{Av}_{\psi, *}} \scc^{U_j, \psi}.$$Thus to prove the lemma it suffices to verify the commutativity of the following diagrams, for any $j \geqslant i \geqslant 1$: \begin{equation}\label{chchd}\xymatrix{\scc^{U_j, \psi} \ar[r]^{\operatorname{Oblv}} \ar[d]_{\delta_h \star - } & \scc^{U_{i}, \psi} \ar[d]^{\delta_h \star - } & & \scc^{U_i, \psi} \ar[r]^{\operatorname{Av}_{\psi, *}} \ar[d]_{\delta_h \star - } &  \scc^{U_j, \psi} \ar[d]^{\delta_h \star - }  \\ \scc^{U_j, h\psi} \ar[r]^{\operatorname{Oblv}} & \scc^{U_{i}, h\psi} & & \scc^{U_i, h \psi} \ar[r]^{\operatorname{Av}_{h \psi, *}} & \scc^{U_j, h\psi}. }\end{equation}For the left hand diagram, this is tautological, and for the right hand diagram this follows from Equation \eqref{exps}, applied to $\underline{e}^{-\psi}$ on $U_j$. Moreover, a simple refinement argument using Equation \eqref{chchd} shows the resulting isomorphisms are independent of the choice of subgroups $U_j, j \geqslant 1$. \end{proof}
Applying the lemma to Kac--Moody representations, we obtain canonical isomorphisms, for any nondegenerate characters $\psi, \psi'$ of $N$ and level $\kappa'$: $$\cW_{\kappa'}^\psi\operatorname{-mod} \simeq \hat{\fg}_{\kappa'}\operatorname{-mod}_{LN, \psi} \simeq \hat{\fg}_{\kappa'}\operatorname{-mod}_{LN, \psi'} \simeq \cW_{\kappa'}^{\psi'}\operatorname{-mod}.$$Concatenating this with Equation \eqref{bigdeal} yields the claimed duality. \end{proof}

\end{theo}

We finish this subsection by checking some compatibilites satisfied by the isomorphisms between $\cw_\kappa^\psi\operatorname{-mod}$, for varying $\psi$. 

\begin{pro} Let $\scc$ be a $D_\kappa(LG)$ module. Then the isomorphisms of Lemma \ref{lchp} are compatible with the isomorphism between Whittaker invariants and coinvariants. That is, the following diagram commutes:

$$\xymatrix{\scc_{LN, \psi} \ar[r]^{\eqref{rl}} \ar[d]_{\eqref{chp}} & \scc^{LN, \psi} \ar[d]^{\eqref{chp}} \\ \scc_{LN, h\psi} \ar[r]^{\eqref{rl}} & \scc^{LN, h \psi}.  }$$
\end{pro}

\begin{proof} Let $U$ be a prounipotent subgroup of $LG$ stable under conjugation by $H$, and consider a character $\chi$ of $U$. Then, the analog of Equation \eqref{chch} remains true, with a nearly identical proof. Similarly, if $U'$ is another prounipotent subgroup of $LG$ containing $U$, and $\chi$ is the restriction of a character $\chi'$ of $U'$, then the analog of Equation \eqref{chchd} remains true. 

Briefly, since every functor appearing in the construction of \eqref{rl} is built from (co)limits of composites of averaging and forgetful functors, the required compatibility follows from \eqref{chchd}. 

More carefully, recall that the functors $i_{n,m, *}, i_{n,m, !},$ and $i_{n,m}^!$, cf. Subsection \ref{adwc}, are composites of averaging and forgetful functors. Applying $\delta_h \star - $ termwise to the appearing colimits and limits yields by \eqref{chchd}: $$\delta_h \star -: \varinjlim_{i_{n,m, *}} \scc^{I_n, \psi} \simeq \varinjlim_{i_{n,m,*}} \scc^{I_n, h\psi} \quad \quad \varinjlim_{i_{n,m, !}} \scc^{I_n, \psi} \simeq \varinjlim_{i_{n,m,!}} \scc^{I_n, h\psi}  \quad \quad \varprojlim_{i_{n,m}^!} \scc^{I_n, \psi} \simeq \varprojlim_{i_{n,m}^!} \scc^{I_n, h \psi}.$$
Writing $I_n^+$ for the subgroup $I_n \cap LN$, we may use the $I_n^+$ as the subgroups $U_n, n \geqslant 1$, appearing in Lemma \ref{lchp}. Concatenating the steps in the adolescent Whittaker construction, we will show each square in the following diagram commutes.

$$\xymatrix{\underset{\operatorname{Av}_{\psi, *}}\varinjlim \scc^{I_n^+, \psi} \ar[d]^{\delta_h \star - } \ar@{}[dr]|{(1)} & \ar@{}[dr]|{(2)} \ar[d]^{\delta_h \star - } \ar[l]_{\operatorname{Oblv}} \underset{{i_{n,m *}}}\varinjlim \scc^{I_n, \psi} \ar[r]^{[-]} & \ar@{}[dr]|{(3)}\ar[d]^{\delta_h \star - }\underset{{i_{n,m !}}}\varinjlim \scc^{I_n, \psi} \ar@{-}[r]^{\sim} & \ar@{}[dr]|{(4)}\ar[d]^{\delta_h \star - } \underset{i_{n,m}^!}\varprojlim \scc^{I_n, \psi} & \ar[d]^{\delta_h \star - } \ar[l]_{\operatorname{Av}_{\psi, *}} \underset{{\operatorname{Oblv}}}\varprojlim \scc^{I_n^+, \psi} \ar@{-} \\\underset{\operatorname{Av}_{h\psi, *}}\varinjlim \scc^{I_n^+, h\psi} & \ar[l]_{\operatorname{Oblv}} \underset{{i_{n,m *}}}\varinjlim \scc^{I_n, h\psi} \ar[r]^{[-]} & \underset{{i_{n,m !}}}\varinjlim \scc^{I_n, h\psi} \ar@{-}[r]^{\sim} & \underset{i_{n,m}^!}\varprojlim \scc^{I_n, h\psi} & \ar[l]_{\operatorname{Av}_{h\psi, *}} \underset{{\operatorname{Oblv}}}\varprojlim \scc^{I_n^+, h\psi} \ar@{-}  }$$We claim that the squares (1), (2), and (4) commute before passing to (co)limits. For the squares (1) and (4), this follows from Equation \eqref{chchd}. For the square (2), this is tautological, since both horizontal arrows are the same cohomological shift.

For (3), it suffices to verify, for any $k \geqslant j \geqslant 1$, the commutativity of the outer square of: 
$$\xymatrix{\scc^{I_j, \psi} \ar[d]^{\delta_h \star - } \ar[r]^{\operatorname{ins}} & \ar[d]^{\delta_h \star - }\underset{{i_{n,m!}}}\varinjlim \scc^{I_n, \psi} \ar@{-}[r]^\sim &\ar[d]^{\delta_h \star - } \underset{i_{n,m}^!} \varprojlim \scc^{I_n, \psi} \ar[r]^{\operatorname{ev}} & \ar[d]^{\delta_h \star - }\scc^{I_k, \psi}\\ \scc^{I_j, h\psi} \ar[r]^{\operatorname{ins}} & \underset{{i_{n,m!}}}\varinjlim \scc^{I_n, h\psi} \ar@{-}[r]^\sim & \underset{i_{n,m}^!} \varprojlim \scc^{I_n, h\psi} \ar[r]^{\operatorname{ev}} & \scc^{I_k, h\psi}}$$By the fully-faithfulness of the $i_{n,m!}$, the horizontal composites are $i_{j,k!}$, whence the claim again follows from Equation \eqref{chchd}. \end{proof}

As a second compatibility, note that the $H$ integrability of $\mathbb{V}_\kappa$ yields an isomorphism of $L\fn$ modules $\mathbb{V}_\kappa \otimes \mathbb{C}_\psi \simeq \mathbb{V}_\kappa \otimes \mathbb{C}_{h \psi}$ for any $h \in H(\mathbb{C})$. Taking semi-infinite cohomology, we obtain a canonical isomorphism $ \cW_\kappa^{\psi} \simeq \cW_\kappa^{h \psi}. $ This tautologically gives an identification of their bounded below derived categories of representations \begin{equation} \cW_\kappa^{\psi}\operatorname{-mod}^+ \simeq \cW_\kappa^{h \psi}\operatorname{-mod}^+.\label{bb}\end{equation}This exchanges the compact generators for the renormalized unbounded categories, yielding 
\begin{equation}\cW_\kappa^\psi\operatorname{-mod} \simeq \cW_\kappa^{h \psi}\operatorname{-mod}.\label{w=w}\end{equation}
The second compatibility we would like to record is between this isomorphism and affine Skyrabin: 

\begin{pro}The following diagram commutes: $$\xymatrix{\gkmod_{LN, \psi} \ar[r]^{\eqref{affs}} \ar[d]_{\eqref{chp}} & \cW_\kappa^{\psi}\operatorname{-mod} \ar[d]^{\eqref{w=w}} \\ \gkmod_{LN, h\psi} \ar[r]^{\eqref{affs}} & \cW_\kappa^{h \psi}\operatorname{-mod}.} $$\end{pro}
\begin{proof}This follows from the fact that $\delta_h \star -$ on $\gk\operatorname{-mod}$ is restriction of representations along $\operatorname{Ad}_{h^{-1}}$. \end{proof}

\subsection{Semi-infinite cohomology for the $\cW$-algebra}

Recall that in the previous subsection we saw a duality between  $\cW_\kappa\operatorname{-mod}$ and $\cw_{-\kappa + 2\kappa_c}\operatorname{-mod}.$ 

\begin{defn} The semi-infinite cohomology functor\begin{equation}\label{sinf}C^{\frac{\infty}{2} + *}( - \otimes - ): \cW_\kappa\operatorname{-mod} \otimes \cW_{-\kappa + 2\kappa_c} \rightarrow \operatorname{Vect}\end{equation}is the pairing induced by Theorem \ref{ffd}. 
\end{defn}
In the next remark, we provide some orienting discussion. 

\begin{re} 

\begin{enumerate}
    \item Recall $h^\vee = -\frac{\kappa_c}{\kappa_b}$ denotes the dual Coxeter number of $\fg$. Recall that the central charge of $\cW_\kappa, \kappa = k \kappa_b$, is given by: $$c(\kappa) = \dim \fg \frac{ k}{k + h^\vee} - \dim \fg + \operatorname{rk} \fg + 24 \langle \rho, \check{\rho} \rangle - 12(k + h^\vee)\kappa_b(\check{\rho}, \check{\rho}).$$Since $-\kappa + 2 \kappa_c = (-k -2h^\vee) \kappa_b$, it follows that the central charges for complementary levels are again complementary:$$c(\kappa) + c(-\kappa + 2\kappa_c) = 2 \operatorname{rk} \fg  + 48 \langle \rho, \check{\rho} \rangle.$$Let us write $c_{-\operatorname{Tate}} := 2 \operatorname{rk} \fg + 48 \langle \rho, \check{\rho} \rangle$ for their common sum.

\item Let $\fg$ be simply laced. In this case, we may equivalently write $$c(\kappa) = \operatorname{rk} \fg - \operatorname{rk} \fg (h^\vee)(h^\vee + 1) \frac{(k + h^\vee - 1)^2}{k + h^\vee}.$$If we consider the resulting map $$\mathbb{P}^1 \setminus \{ -h^\vee, \infty \} \rightarrow \mathbb{P}^1 \setminus \{ \infty \}, \quad \quad k \mapsto c(k \kappa_b),$$the $\cW$-algebras corresponding to levels with the same central charge are identified by the Feigin--Frenkel isomorphisms:$$\cw_{(-h^\vee + \epsilon)\kappa_b} \simeq \cw_{(-h^\vee + \epsilon^{-1})\kappa_b}, \quad \quad \epsilon \in \mathbb{C}.$$
Thus, we may unambiguously write $\cW_c$ for the $\cW$-algebra associated to $\fg$ of central charge $c$, and may rewrite \eqref{sinf} as$$\cw_c \operatorname{-mod} \otimes \cw_{c_{-\operatorname{Tate}} - c}\operatorname{-mod} \rightarrow \operatorname{Vect}, \quad \quad c \in \mathbb{C}.$$This latter parametrization was used in the known cases of semi-infinite cohomology, namely for $\fg = \mathfrak{sl}_2, \mathfrak{sl}_3$, with $c_{-\operatorname{Tate}} = 26, 100,$ respectively.  
\item Our functor \eqref{sinf} takes the shape of the so-called {\em non-critical} BRST reduction anticipated in the conformal field theory literature. For $\cW_\kappa$ with $c(\kappa) = c_{-\operatorname{Tate}}$,\footnote{In the reference \cite{b}, $c_{-\operatorname{Tate}}$ is equivalently written as $2\sum_s(6s^2 -6s + 1)$, where $s$ runs over the conformal dimensions of the standard generators for $\cW$, i.e. the degrees of the fundamental invariants for $\fg$.} it was anticipated there should be a {\em critical} reduction functor: $$C^{\frac{\infty}{2} + *}(-): \cW_\kappa\operatorname{-mod} \rightarrow \operatorname{Vect}.$$This can be done by a similar method, as we will explain in a subsequent publication. 
\end{enumerate}\label{rere}
\end{re}

We bring two questions to the attention of the reader. Firstly, as suggested by Remark \ref{rere}(2), there should be a compatibility between semi-infinite cohomology for $\cW$-algebras and Feigin--Frenkel duality. Second, it would be good to establish the compatibility between the usual semi-infinite cohomology for Virasoro and one constructed in this subsection. 

For both, one can show as proof of concept that the dualities coming from the two different constructions of semi-infinite cohomology send the standard compact generators to the same objects, up to isomorphism. 
\subsection{An example: chiral differential operators} In this subsection, $\fg = \mathfrak{sl}_2$ and $\kappa$ is irrational, i.e.  $\kappa \in \mathbb{C} \kappa_b \setminus \mathbb{Q} \kappa_b$. We now address a problem raised by I. Frenkel and Styrkas in Remark 7 of \cite{f1}, i.e. to give a direct explanation of why the semi-infinite cohomology of the CDO for $\operatorname{Vir}_\kappa$ matched the semi-infinite cohomology of the CDO for $\gk$. As anticipated in {\em loc. cit.}, we show this may be done via Drinfeld--Sokolov reduction. 

Since we will return to this question and others from {\em loc. cit.} in more generality elsewhere, we only provide a sketch here. Let $P^+ \simeq \mathbb{Z}^{\geqslant 0}$ index the dominant integral weights of $\fg$, and for $\lambda \in P^+$ write $\mathbb{V}_{\kappa, \lambda}$ for the corresponding Weyl module of $\gk$. As a $\gk \oplus \gkm$ module, the Chiral differential operators for $\gk$ decompose as: $$\operatorname{CDO} \simeq \bigoplus_{\lambda \in P^+} \mathbb{V}_{\kappa, \lambda} \otimes \mathbb{V}_{-\kappa + 2\kappa_c, - w_\circ \lambda}.$$Using Kac--Moody duality, we may calculate its semi-infinite cohomology as: $$\csi(\hat{\fg}_{2\kappa_c}, L^+ \fg, \operatorname{CDO}) \simeq \bigoplus_{\lambda \in P^+} \operatorname{Hom}_{\gk\operatorname{-mod}}( \mathbb{V}_{\kappa, \lambda}, \mathbb{V}_{\kappa, \lambda}) \simeq \operatorname{Fun}(G /\!\!/ G) \otimes \mathbb{C}^{h \fg},$$where $\mathbb{C}^{h \fg}$ denotes the Lie algebra cohomology $\operatorname{Hom}_{\fg}(\mathbb{C}, \mathbb{C})$. 
The chiral differential operators for $\operatorname{Vir}_\kappa$ are simply $\Psi \boxtimes \Psi \operatorname{CDO},$ where $\Psi$ denotes Drinfeld--Sokolov reduction.\footnote{We should mention that $\Psi \operatorname{CDO}$ depends on $\kappa$ and not simply $c(\kappa)$. More generally, a basic fact of life is that the two `Kazhdan-Lusztig categories' in $\cW_\kappa\operatorname{-mod}$ coming $\fg, \fg^L$ do {\em not} coincide. Rather, as explained to us by Creutzig, they form two `axes' of a remarkable two parameter family of representations.} Accordingly, $$\Psi \operatorname{CDO} \simeq \bigoplus_{\lambda \in P^+} \Psi \mathbb{V}_{\kappa, \lambda} \otimes \Psi \mathbb{V}_{-\kappa + 2\kappa_c, \lambda}.$$Using categorical Feigin--Fuchs duality, we have $$\csi( \operatorname{Vir}_{26}, \Psi \operatorname{CDO}) \simeq \bigoplus_{\lambda \in P^+} \operatorname{Hom_{Vir_\kappa\operatorname{-mod}}}( \Psi \mathbb{V}_{\kappa, \lambda}, \Psi \mathbb{V}_{\kappa, \lambda}).$$It therefore suffices to argue that the natural map $$\operatorname{Hom}(\mathbb{V}_{\kappa, \lambda}, \mathbb{V}_{\kappa, \lambda}) \rightarrow \operatorname{Hom}(\Psi \mathbb{V}_{\kappa, \lambda}, \Psi \mathbb{V}_{\kappa, \lambda})$$is an equivalence. But since $\kappa$ is generic, in fact the corresponding block of monodromic category $\OO$ for $\gk$ is sent isomorphically onto the corresponding block of monodromic category $\OO$ for $\operatorname{Vir}_\kappa$ by $\Psi$. This follows from the localization theorem alluded to at the end of Section \ref{two}, see also Subsection \ref{cbl}.

\section{Jordan-H\"older content and highest weight filtrations in Category $\OO$}

Recall the definition of Category $\OO$ from Subsection \ref{cato}. We remind two basic properties of Category $\OO$ for an affine Lie algebra. First, at a positive level $\kappa$ many objects will not have finite length. Nonetheless, one can make sense of the Jordan--H\"older content of an object, i.e. express the formal character of a module as a locally finite sum of simple characters. Second, any object of $\OO$ admits an ascending filtration with successive quotients highest weight modules. In this section, we set up the analogous theory for $\mathscr{W}$-algebras, for $\mathscr{W}_3$ see \cite{bw3}. 

\subsection{Jordan-H\"{o}lder content} As a first observation, note that since $\mathscr{W}$ is $\Z^{\geqslant 0}$ graded by $L_0$, we have the following coarse decomposition of $\OO$:

\begin{lemma} Let $M$ be any object of $\OO$, and write $M = \bigoplus_{d \in \C} M_d$ for its $L_0$ eigenspace decomposition as in Definition \ref{defoglarg}. For $\gamma \in \C/\Z$, consider the subspace: $$M_\gamma := \bigoplus_{d \in \gamma + \Z} M_d.$$Then $M_\gamma$ is a submodule of $M$, and we have: $M = \bigoplus_{\gamma \in \C  / \Z} M_\gamma.$

\label{break}
\end{lemma}
 
Thus, to define the Jordan-H\"older content of $M$, it suffices to consider $M = M_\gamma$ for some $\gamma \in \C/\Z$. Fix a lift $\dot{\gamma} \in \C$ of $\gamma$. For $\lambda \in \Wf \backslash \h^*$, write $\langle L_0, \lambda \rangle$ for its lowest energy, i.e. the evaluation of $\lambda$ on the element of $\operatorname{Zhu}(\cW_\kappa)$ corresponding to the conformal vector $\omega$. 

\begin{lemma} An object $M = M_\gamma$ admits a finite filtration $$0 = M_0 \subset M_1 \subset \cdots \subset M_{n-1} \subset M_n = M$$such that each successive quotient $M_i/M_{i-1}$ is either (i) a simple module $L(\lambda_i)$, with $ -\langle L_0, \lambda_i \rangle \in \dot{\gamma} + \Z^{\geqslant 0}$, or (ii) has $-L_0$ eigenvalues in $\dot{\gamma} + \Z^{< 0}$. 
\label{ffil} \end{lemma}

\begin{proof} By assumption, $D := \dim \bigoplus_{n \in \Z^{\geqslant 0}} M_{\dot{\gamma} + n} < \infty.$ Take $n$ maximal for which $M_{\dot{\gamma} + n}$ is nonzero. Then $\operatorname{Zhu}(\mathscr{W})$ acts on $M_{\dot{\gamma} + n}$, and we may take an eigenvector $v$ with eigenvalue $\lambda \in W_{\operatorname{f}} \backslash \h^*$. This induces a map $M(\lambda) \rightarrow v$, whose image we call $M'$. There is a short exact sequence $0 \rightarrow N' \rightarrow M' \rightarrow L(\lambda) \rightarrow 0$ which we use to form a filtration $N' \subset M' \subset M.$ By induction on $D$, we may produce filtrations of $N'$ and $M/M'$ of the desired form. The induced filtration of $M$ satisfies the conditions of the lemma. \end{proof}

We would like to say for  $\lambda \in W_{\operatorname{f}} \backslash \h^*,$ $-\langle L_0, \lambda \rangle \in \dot{\gamma} + \Z^{\geqslant 0}$, that $[M: L(\lambda)]$ should be the number of times in a filtration as above the successive quotient is isomorphic to $L(\lambda)$. 

To do so, consider $\OO_{\gamma}$, the full subcategory of $\OO$ consisting of objects with $-L_0$ eigenvalues in $\gamma + \Z$. Within $\OO_\gamma$, consider $\OO_{<\dot{\gamma}}$, the full subcategory of $\OO$ consisting of objects with $-L_0$ eigenvalues in $\dot{\gamma} + \Z^{ < 0}.$ By construction, we have:

\begin{lemma} $\OO_{< \dot{\gamma}}$ is a thick subcategory of $\OO_{\gamma}$, i.e. is a full abelian subcategory closed under subquotients and extensions. \end{lemma}
As with any thick subcategory of an abelian category, one can form the quotient abelian category $\OO_\gamma / \OO_{< \dot{\gamma}}$. The objects are again those of $\OO_\gamma$, but the morphisms between $M$ and $N$ are the direct limit over $\Hom(M', N/N')$, where $M'$ is a subobject of $M$ with $M/M' \in \OO_{< \dot{\gamma}}$ and $N' \in \OO_{< \dot{\gamma}}$ is a subobject of $N$.

\begin{pro} In the quotient category $\OO_\gamma / \OO_{< \dot{\gamma}}$ every object is of finite length, and the isomorphism classes of irreducibles are given by the $L(\lambda), -\langle L_0,  \lambda \rangle \in \dot{\gamma} + \Z^{\geqslant 0}.$
 \end{pro}

\begin{proof} It follows from Lemma \ref{ffil} that every object is of finite length, and any simple object is of the form $L(\lambda), -\langle L_0,  \lambda \rangle \in \dot{\gamma} + \Z^{\geqslant 0}$. From the explicit form of morphisms, and using the simplicity of the $L(\lambda)$, it is easy to see the $L(\lambda), -\langle L_0,  \lambda \rangle \in \dot{\gamma} + \Z^{\geqslant 0},$ remain mutually non-isomorphic. \end{proof}

With this, we can prove:

\begin{theo} Fix $M \in \OO_\gamma$ and $L(\lambda)$ with $- \langle L_0, \lambda \rangle \in \gamma + \Z$. If we pick a $\dot{\gamma}$ with $-\langle L_0, \lambda \rangle \in \dot{\gamma} + \Z^{\geqslant 0}$, and a filtration of $M$ as in Lemma \ref{ffil}, then the number $[M:L(\lambda)]$ of times $M_i/M_{i-1}$ is isomorphic to $L(\lambda)$ is independent of the choice of $\dot{\gamma}$ and the filtration.

\end{theo}

\begin{proof} For fixed $\dot{\gamma}$, the independence from the choice of filtration follows from the fact that $[M:L(\lambda)]$ is really the Jordan-H\"older content in the quotient category $\OO_\gamma / \OO_{< \dot{\gamma}}$. To compare the number $[M:L(\lambda)]$ obtained from $\dot{\gamma}$ and $\dot{\gamma} - n, n \in \Z^{\geqslant 0}$, one can refine a filtration for $\dot{\gamma}$ as in Lemma \ref{ffil} to obtain one for $\dot{\gamma} - n$. \end{proof}

For a general object $M$ of $\OO$, and $L(\lambda)$ with $-\langle L_0, \lambda \rangle \in \gamma + \Z$, we define $[M: L(\lambda)] := [M_\gamma: L(\lambda)]$, where $M_\gamma$ was defined in Lemma \ref{break}. We now collect the basic properties of this Jordan-H\"older content:

\begin{pro} For a short exact sequence $0 \rightarrow M' \rightarrow M \rightarrow M'' \rightarrow 0$ and any $\lambda \in W_{\operatorname{f}} \backslash \h^*$, we have $[M: L(\lambda)] = [M': L(\lambda)] + [M'': L(\lambda)].$ \label{jh1} \end{pro}

\begin{pro} An object $M$ of $\OO$ is zero if and only if $[M: L(\lambda)] = 0$ for all $\lambda \in W_{\operatorname{f}} \backslash \h^*$. 
 \label{jh2}
\end{pro}

\subsection{(Co-)Highest weight filtrations} Recall the basic operation of picking a highest weight vector in an object of Category $\OO$, which already appeared in Lemma \ref{ffil}. Iterating this, we obtain:

\begin{pro} Suppose $M = M_\gamma$ for some $\gamma \in \C/\Z$. Then $M$ admits an ascending filtration: $$0 = M_0 \hookrightarrow M_1 \hookrightarrow M_2 \hookrightarrow \cdots, \quad \quad \varinjlim_i M_i = M,$$such that each successive cokernel $M_i/M_{i-1}, i \geqslant 1,$ is a highest weight module. 
 
 \label{hwfil}
\end{pro}

\begin{proof} Begin as in Lemma \ref{ffil} to obtain $M'$, a highest weight submodule of $M$ whose highest weight line lies in the maximal nonzero $M_d, d \in \gamma + \Z$. We set $M_1 = M'$, and apply the same construction to $M/M_1$ to produce $M_2$, etc. Since we use a maximal nonzero $M_d$ at each step, and the weight spaces $M_d, d \in \C$, are all finite dimensional, it is clear this construction satisfies the conditions of the proposition. 
\end{proof}

By applying the duality $\mathbf{D}$ on Category $\OO$, we obtain a dual form of Proposition \ref{hwfil}.

\begin{pro} Suppose $M = M_\gamma$ for some $\gamma \in \C/\Z$. Then $M$ admits a descending filtration: $$ \cdots \twoheadrightarrow A_{-2} \twoheadrightarrow A_{-1} \twoheadrightarrow A_0 = 0, \quad \quad \varprojlim_i A_i = M,$$such that each successive kernel $\ker(A_{i} \rightarrow A_{i+1}), i \leqslant -1$, is a co-highest weight module. 
 
 \label{chwfil}
\end{pro}

\begin{re} It may be psychologically helpful for the reader to note that in the the inverse limit of Proposition \ref{chwfil}, for each $M_d, d \in \C,$ the inverse limit $\varprojlim_i (A_i)_d$ stabilizes after finitely many steps. 
\end{re}

\begin{re}The results and proofs in this section should apply verbatim to any $\mathbb{Z}^{\geqslant 0}$ graded vertex algebra which is finitely strongly generated, cf. \cite{fbz}. 
\end{re}

\section{The abstract linkage principle} \label{slinkies}
In this section and its sequel, we will  canonically decompose Category $\OO$ for $\cW_\kappa$ as a direct sum over images of $W$ orbits in $W_{\operatorname{f}} \backslash \h^*$. For nonintegral weights and levels the blocks should, as usual,  refine the above decomposition. In this first section, we give a general abstract decomposition of $\OO$ into blocks. For Kac--Moody algebras this is done in \cite{dgk}, \cite{moop} by similar means, albeit with less use of co-highest weight filtrations. In Section \ref{clinkies}, we will then provide a Coxeter theoretic description.

We now introduce notation. For $\mu, \nu \in \Wf \backslash \h^*,$  write $\mu \prec \nu$ if $L(\mu)$ is a subquotient of $M(\mu)$. Consider the finest partition$$\Wf \backslash \h^* = \bigsqcup_{\alpha} \mathscr{P}_\alpha$$such that if $\nu \in \mathscr{P}_\alpha$ and $\nu \prec \mu$, then $\mu \in \mathscr{P}_\alpha$. For fixed $\mathscr{P}_\alpha$, write $\OO_\alpha$ for the full subcategory of $\OO$ consisting of objects $M$ all of whose simple subquotients have highest weights in $\mathscr{P}_\alpha$. We will refer to $\mathscr{P}_\alpha$ as a {\em linkage class} and to $\OO_\alpha$ as the corresponding {\em block}. Our goal is to prove the following theorem, which gives the block decomposition of $\OO$.

\begin{theo} Every object $M$ of $\OO$ canonically decomposes as a sum $\bigoplus_\alpha M_\alpha$, where $M_\alpha$ lies in $\OO_\alpha$. 
\label{blocks}
\end{theo}

We will obtain Theorem \ref{blocks} from the following theorem, which we turn to next.

\begin{theo} For objects $M \in \OO_\alpha$ and $A \in \OO_{\alpha'},$ $\alpha \neq \alpha'$, $\operatorname{Ext}^1(M, A) = 0.$
\label{extv}
\end{theo}

We now begin the proof of Theorem \ref{extv}. 
\begin{lemma} For objects $M \in \OO_\alpha$ and  $A \in \OO_{\alpha'}$, $\Hom(M, A) = 0.$ In particular, if an extension of $M$ by $A$ splits, it does so uniquely. 
\label{nohom} \label{unsp}
\end{lemma}

\begin{proof} The first point follows from considering the Jordan--H\"older content of the image of a morphism, which would necessarily vanish, cf. Propositions \ref{jh1}, \ref{jh2}. The second follows since the set of splittings is a torsor over $\Hom(M, A)$.   \end{proof}

Next, we check the desired $\Ext$ vanishing for Verma and co-Verma modules.

\begin{lemma} For $\lambda$ and  $\nu \in W_{\operatorname{f}} \backslash \h^*, \lambda \neq \nu$, $\Ext^1(M(\lambda), A(\nu)) = 0.$
\label{vcv}
\end{lemma}

\begin{proof} Using the duality $\mathbf{D}$ on $\OO$, we have a canonical isomorphism $$\Ext^1( M(\lambda), A(\nu) ) \simeq \Ext^1( M(\nu), A(\lambda) ).$$Since $q$-characters are unchanged by duality, by using the above isomorphism we may assume without loss of generality that $- \langle L_0, \nu \rangle \notin - \langle L_0,  \lambda \rangle + \Z^{> 0}$. Suppose we have an extension $0 \rightarrow A(\nu) \rightarrow E \rightarrow M(\lambda) \rightarrow 0.$ Writing $d := - \langle L_0, \lambda \rangle$, i.e. the highest $-L_0$ eigenvalue of $M(\lambda)$, we have an exact sequence
\begin{equation} 0 \rightarrow A(\nu)_d \rightarrow E_d \rightarrow M(\lambda)_d \rightarrow 0. \label{toppiece}\end{equation}By our assumption, $E_{d + n} = 0$ for $n \in \Z^{\geqslant 0}$, hence $E_d$ consists of singular vectors. Viewing Equation \eqref{toppiece} as a short exact sequence of $\operatorname{Zhu}(\mathscr{W})$ modules, since $\lambda \neq \nu$ the sequence is split uniquely. By the universal property of $M(\lambda)$, the corresponding morphism $M(\lambda) \rightarrow E$ is again a splitting. \end{proof}

We now deduce the desired statement by d\'evissage, as explained in the following lemmas. 

\begin{lemma} For a highest weight module $M \in \OO_{\alpha}$ and $A(\nu) \in \OO_{\alpha'}$, $\Ext^1(M, A(\nu)) = 0.$
\label{int1}
\end{lemma}

\begin{proof} By assumption, we have an exact sequence $0 \rightarrow K \rightarrow M(\lambda) \rightarrow M \rightarrow 0,$ where $M(\lambda)$ and  $K$ are objects of $\OO_{l}$. This gives an exact sequence: $$\Hom(K, A(\nu) ) \rightarrow \Ext^1( M, A(\nu) \rightarrow \Ext^1( M(\lambda), A(\nu)),$$where $\Hom(K, A(\nu))$ and $\Ext^1( M(\lambda), A(\nu))$ vanish by Lemmas \ref{nohom} and \ref{vcv}, respectively.   \end{proof}

\begin{lemma} For a highest weight module $M \in \OO_{\alpha}$ and a co-highest weight module $A \in \OO_{\alpha'}$, $\Ext^1(M, A) = 0.$
\end{lemma}
\begin{proof} Similarly to the proof of Lemma \ref{int1}, we use an exact sequence $$0 \rightarrow A \rightarrow A(\nu) \rightarrow C \rightarrow 0,$$and finish by the statements of Lemmas \ref{nohom} and \ref{int1}. \end{proof}
We deduce: 
\begin{lemma} If $M \in \OO_{\alpha}$ has a finite filtration with successive quotients highest weight modules, and $A \in \OO_{\alpha'}$ is a co-highest weight module, then $\Ext^1(M, A) = 0.$ \label{int3} \end{lemma}

\begin{pro} If $M \in \OO_{\alpha}$ is arbitrary, and $A \in \OO_{\alpha'}$ is a co-highest weight module, then $\Ext^1(M,A) = 0.$
 \label{int5}
\end{pro}

\begin{proof} Write $M$ as a colimit $\varinjlim_i M_i$ as in Proposition \ref{hwfil}, and consider an extension: $$0 \rightarrow A \rightarrow E \rightarrow \varinjlim_i M_i \rightarrow 0.$$We will split it on the right.
For fixed $i$, consider the pullback: $$0 \rightarrow A \rightarrow E_i \rightarrow M_i \rightarrow 0.$$To split $E$ it suffices to produce splittings $s_i: M_i \rightarrow E_i$ which are compatible as we vary $i$. But $M_i$ by definition admits a finite filtration with successive quotients highest weight modules, whence by Lemma \ref{int3} admits a splitting $s_i$. If $i \leqslant j$, then the pullback of $s_j$ produces a splitting of $E_i$. Hence the compatibility of the $s_i$ follows from their uniqueness, cf. Corollary \ref{unsp}. \end{proof}

Similarly to Lemma \ref{int3}, we deduce:

\begin{lemma} If $M \in \OO_{\alpha}$ is arbitrary, and $A \in \OO_{\alpha'}$ has a finite filtration with successive quotients co-highest weight modules. Then $\Ext^1(M, A) = 0.$
\label{int4}
\end{lemma}

Finally, an argument dual to that of Proposition \ref{int5} finishes the d\'evissage.
\begin{proof}[Proof of Theorem \ref{extv}] Write $A$ as an inverse limit $\varprojlim_i A_i$ as in Proposition \ref{chwfil}. Given an extension $E \in \Ext^1(M,A)$, one splits it on the left by giving compatible splittings of the pushouts$$0 \rightarrow A_i \rightarrow E_i \rightarrow M \rightarrow 0,$$as in the proof of Proposition \ref{int5}. \end{proof}

\begin{proof}[Proof of Theorem \ref{blocks}]Let $M$ be an arbitrary object of $\OO$. By Lemma \ref{break}, we may reduce to the case $M = M_\gamma, \gamma \in \C/\Z$. Write $M$ as a colimit $\varinjlim_i M_i$ as in Proposition \ref{hwfil}. For each $i$, $M_i$ has a finite filtration with successive quotients highest weight modules. By Theorem \ref{extv}, $M_i$ accordingly has a finite direct sum decomposition of the desired form, which we write as  $M_i = \bigoplus_{\alpha} M_{i,\alpha}.$ By Lemma \ref{nohom}, the morphism $M_i \rightarrow M_j$ decomposes as a sum of morphisms $M_{i,\alpha} \rightarrow M_{j,\alpha}$. Accordingly, we have $$M \simeq \varinjlim_i M_i \simeq \varinjlim_i \bigoplus_{\alpha} M_{i, \alpha} \simeq \bigoplus_{\alpha} \varinjlim_i M_{i, \alpha},$$which has the desired form. \end{proof}

\begin{cor} For $M \in \OO_{\alpha}$ and $A \in \OO_{\alpha'}$, $\alpha \neq \alpha'$, we have: $$\Ext^i_\OO(M, A) = 0, \quad \quad i \geqslant 0.$$ \end{cor}
\begin{re}Since we are working with a monodromic Category $\OO$, it would be good to determine whether the map $\OO \rightarrow \cW\operatorname{-mod}^\heartsuit$ is derived fully faithful. \end{re}

\begin{re} A similar argument applies, {\em mutatis mutandis}, to any $\mathbb{Z}^{\geqslant 0}$ graded vertex operator algebra which is finitely strongly generated. \end{re}

\section{The Coxeter theoretic linkage principle} \label{clinkies}
Let us recall that for Kac--Moody Lie algebras one can prove an abstract linkage principle, as was done in the previous section. This reduces the classification of blocks to a question about the Jordan-H\"older content of Verma modules. The latter question in turn has a beautiful answer in terms of the dot action of the Weyl group on $\h^*$, and is the starting point for the Kazhdan-Lusztig combinatorics. 

In this section, we will perform the analogous classification for $\cW$. Whereas for $\gk$ blocks were parametrized by the cosets of a Coxeter group modulo a parabolic subgroup, for $\cW_\kappa$ the blocks are parmetrized by double cosets of a Coxeter group modulo two parabolic subgroups. Loosely, the one sided quotient is what one had prior to Drinfeld--Sokolov reduction, and the new quotient comes from the projection $\h^* \rightarrow W_{\operatorname{f}} \backslash \h^*$. 
\subsection{The linkage principle: recollections for Kac--Moody} Let $\kappa$ be an arbitrary noncritical level. We now review the classification of blocks for $\gk$. 

Let $\lambda \in \h^*_\kappa$. The block passing through $\lambda$ will be controlled not always by the full affine Weyl group $W$, but rather a subgroup which we now recall. To do so, we will need some notation. Write $\Phi^\vee_+$ for the positive coroots of $\gk$, $\Phi^\vee_-$ for the negative coroots of $\gk$, and $\Phi^\vee = \Phi_+^\vee \sqcup \Phi^\vee_-$ for the coroots. Write $\Phi^\vee_{re}$ for the real coroots, and similarly $\Phi^\vee_{\pm, re}$ for the positive and negative real coroots. The choice of $\lambda$ affords the subset: $$\Phi^\vee_\lambda := \{ \halpha \in \Phi^\vee : \langle \halpha, \lambda + \hat{\rho} \rangle \in \mathbb{Z} \}.$$Intersecting $\Phi^\vee_\lambda$ with the real and positive or negative coroots yields $$\Phi^\vee_{\lambda, \pm} := \Phi^\vee_\lambda \cap \Phi^\vee_{\pm}, \quad \quad  \Phi^\vee_{\lambda, re} := \Phi^\vee_{\lambda} \cap \Phi^\vee_{re},  \quad \quad  \Phi^\vee_{\lambda, re, \pm} := \Phi^\vee_{\lambda} \cap \Phi^\vee_{\pm, re}.$$

The reflections $s_\alpha$ corresponding to every real coroot in $\Phi^\vee_\lambda$ generate a subgroup $W_\lambda$ of $W$, the {\em integral Weyl group} $W_\lambda := \langle s_{\halpha} \rangle,  \halpha \in \Phi^\vee_{\lambda, re}.$ With this, we have the following well-known

\begin{pro} For $\lambda \in \h_\kappa^\vee$, the block of Category $\OO$ for $\gk$ containing $L(\lambda)$ has highest weights $W_\lambda \cdot \lambda$. \end{pro}

We now recall the Coxeter theoretic parametrization of the block. Within $\Phi^\vee_\lambda$ one has the `simple integral coroots' $$\Pi_\lambda := \{ \halpha \in \Phi_{\lambda, +, re}^\vee:  s_{\halpha} \Phi_{\lambda, +}^\vee \cap \Phi^\vee_{\lambda, -} = -\halpha \}.$$Equivalently, $\Pi_\lambda$ consists of the elements of $\Phi^\vee_{\lambda, +, re}$ which cannot be expressed as the sum of two other elements of $\Phi^\vee_{\lambda, +}$. If we index the simple integral roots by $I_\lambda$, i.e. $$\{ \halpha_i, i \in I_\lambda \} := \Pi_\lambda,$$then it is known that $W_\lambda$ is a Coxeter group with simple reflections $s_i, i \in I_\lambda$. 

The orbit $W_\lambda \cdot \lambda$ will pass through a unique dominant or antidominant weight, depending on the sign of $\kappa$: $$\kappa \notin \kappa_c + \Q^{> 0}: \quad \quad C^-_\kappa := \{ \lambda \in \h^*_\kappa: \langle \lambda + \hat{\rho}, \halpha \rangle \notin \Z^{> 0}, \halpha \in \Phi^\vee_{+, re} \},$$\begin{equation}\label{cp}\kappa \in \kappa_c + \Q^{> 0}: \quad \quad C_\kappa^+ := \{ \lambda \in \h^*_\kappa: \langle \lambda + \hat{\rho}, \halpha \rangle \notin \Z^{< 0}, \halpha \in \Phi^\vee_{+, re} \}.\end{equation}Note that the above dichotomy, while a standard practice in the literature, is partially a convention, as nonintegral blocks may contain both a dominant and antidominant weight, and  one can equally well parametrize $\kappa \notin \Q \kappa_c$ by dominant weights.

Observe also that we can rewrite (anti)dominance in terms of the integral Weyl group: \begin{lemma}\label{obv}We have the equalities $$C^{-}_\kappa = \{ \lambda \in \h^*_\kappa: \langle \lambda + \hat{\rho}, \halpha_i \rangle \in \Z^{\leqslant 0}, i \in I_\lambda \},$$\begin{equation} \label{cpp} C^+_\kappa = \{ \lambda \in \h^*_\kappa: \langle \lambda + \hat{\rho}, \halpha_i \rangle \in \Z^{\geqslant 0}, i \in I_\lambda \}.\end{equation}\end{lemma}\begin{proof}Since $\halpha_i, i \in I_\lambda$ lie in $\Phi^\vee_\lambda$, we have $\langle \lambda + \hat{\rho}, \halpha_i \rangle \in \mathbb{Z}$. It follows that the conditions defining the right hand side of \eqref{cpp} are a subset of those in \eqref{cp}. To see the reverse inclusion, for $\halpha \in \Phi^\vee_{+, re} \setminus \Phi^\vee_\lambda$, $\langle \lambda + \hat{\rho}, \halpha \rangle \notin \mathbb{Z}$, and hence the condition of \eqref{cp} is automatic for such $\halpha$. Moreover, if we know the right hand side of \eqref{cpp}, the fact that \eqref{cp} holds for all $\halpha \in \Phi^\vee_{\lambda, re, +}$ follows from knowing that $\Phi^\vee_{\lambda, +} \subset \Z^{\geqslant 0} \{ \halpha_i, i \in I_\lambda \}.$\end{proof} For $\lambda \in C^\pm_\kappa$, the stabilizer of $\lambda$ in $W_\lambda$, which we denote by $W^\circ_\lambda$, is generated by the simple reflections fixing $\lambda$.

\subsection{The linkage principle: the $\cW$-algebra}
Having recalled the classification of blocks for $\gk$, we now perform the analogous classification for $\cW_\kappa$. 

The main subtlety is that for $\overline{\lambda}$ in $W_{\operatorname{f}} \backslash \h^*_\kappa$, there may be several $W_{\operatorname{f}}$ antidominant weights $\lambda \in \h^*_\kappa$ with image $\overline{\lambda}$. To say how many, write $\Phf$ for the finite coroots, i.e. associated to $\mathfrak{g}$, and consider $$\Phf_{\lambda} := \{ \halpha \in \Phf: \langle \lambda + \rho, \halpha \rangle \in \mathbb{Z} \}.$$We may again consider the subgroup $W_{\operatorname{f}, \lambda} \subset \Wf$ generated by the reflections $s_{\halpha}, \halpha \in \Phf_\lambda$. Then there are $\Wf / W_{\operatorname{f}, \lambda}$ many $\Wf$ antidominant weights in $\Wf \cdot \lambda$, as we develop in the next two lemmas.

\begin{lemma} Write $\wstabf$ for the stabilizer of $\lambda$ in $\Wf$. Then $\wstabf$ is contained in $W_{\operatorname{f}, \lambda}$.  \label{fstab}\end{lemma}

\begin{proof}This is well known, and can be proved by adapting the argument of Lemma \ref{stabaff}.
\end{proof}

\begin{lemma} For each coset $y W_{\operatorname{f}, \lambda} \in \Wf / W_{\operatorname{f, \lambda}}$, there is a unique $W_{\operatorname{f}}$ antidominant weight in $y W_{\operatorname{f}, \lambda} \cdot \lambda$. In particular, under the identification of $\Wf \cdot \lambda \simeq \Wf/\wstabf$, each fibre of the projection onto $\Wf/W_{\operatorname{f}, \lambda}$ contains a unique $\Wf$ antidominant weight. \label{cosets} 
\end{lemma}
\begin{proof} It is straightforward that there is a unique $\Wf$ antidominant weight in $W_{\operatorname{f}, \lambda} \cdot \lambda$, cf. the proof of Lemma \ref{obv}. To finish the first part of the lemma, note that for any $y \in \Wf$, $W_{\operatorname{f}, y \cdot \lambda} = y W_{\operatorname{f}, \lambda} y^{-1}$, and hence $W_{\operatorname{f}, y \cdot \lambda} \cdot (y \cdot \lambda) = y W_{\operatorname{f}, \lambda} \cdot \lambda$. The remainder of the Lemma now follows from Lemma \ref{fstab}.\end{proof}

The previous two lemmas describe $\Wf$ antidominance within a single $\Wf$ orbit. We next need to understand $\Wf$ antidominance within a single $W$ orbit, or an integral Weyl group $W_\lambda$ orbit therein.

\begin{lemma} Let $\lambda \in \h^*_\kappa, \kappa \neq \kappa_c.$ Write $\wstab$ for the stabilizer of $\lambda$ in $W$. Then $\wstab$ is contained in $W_\lambda$. \label{stabaff} 
\end{lemma}

\begin{proof}
It suffices to see that $\wstab$ is generated by the affine reflections $s_{\halpha}, \halpha \in \Phi^\vee_{re},$ it contains. We moreover claim it suffices to check this for a single $\kappa_\circ \neq \kappa_c$. Indeed, the scaling action of $\mathbb{G}_m$ on $\h^*$, centered at $-\rho$, yields $W$ equivariant isomorphisms between all noncritical $\h^*_\kappa$. 

We will take $\kappa_\circ \in \kappa_c + \mathbb{R}^{> 0} \kappa_b$.  Note that for such $\kappa$, 
$\h^*_\mathbb{R}$, the real affine span of the integral weights at that level, is stable under the level $\kappa$ dot action of $W$ on $\h^*$. Let us denote this restricted action by $\h^*_{\mathbb{R}, \kappa}$. Under the decomposition $\h^*_\kappa = \h^*_{\mathbb{R}} \oplus i \h^*_{\mathbb{R}}$, 
$W$ acts diagonally. The first factor transforms as $\h^*_{\mathbb{R}, \kappa}$, and the second transforms as the level zero unshifted action of $W$ on $\h^*_\mathbb{R}$. 

Let us write $\lambda \in \h^*_{\kappa_0} = \lambda_{re} + i \lambda_{im}$. We may compute $\wstab$ by first computing the stabilizer of $\lambda_{re}$, and then taking the subgroup of it which fixes $\lambda_{im}$. By our choice of $\kappa_\circ$, we may replace $\lambda$ by a $W$ translate so that $\lambda_{re}$ lies in the positive alcove: $$C^+ := \{ \nu \in \h^*_{\mathbb{R}}: \langle \nu + \hat{\rho}, \halpha_i \rangle \in \mathbb{R}^{\geqslant 0}, i \in \hat{I} \}.$$The stabilizer of $\lambda_{re}$ is then a parabolic subgroup $W_J$ of $W$, for a proper subset $J \subset \hat{I}$ of the affine simple roots. 

Because $J$ was a proper subset of $\hat{I}$, we may further act by $W_J$ to carry $\lambda_{im}$ into $$C^+_J := \{ \nu \in \h^*_\mathbb{R}: \langle \nu, \halpha_j \rangle \in \mathbb{R}^{\geqslant 0}, j \in J \}.$$It follows that the stabilizer of $\lambda_{im}$ in $W_J$ is generated by the simple reflections $s_j, j \in J$, it contains, as desired. \end{proof}

\begin{lemma} Fix $\lambda \in \h^*_\kappa, \kappa \neq \kappa_c$, and $\nu \in W_\lambda \cdot \lambda$. Then the intersection of $\Wf \cdot \nu$ and $W_\lambda \cdot \lambda$ is (i) acted on transitively by $W_{\operatorname{f}, \lambda}$, and (ii) contains a unique $\Wf$ antidominant weight.\label{inorbits} \end{lemma}

\begin{proof}We start with (i), i.e. that $$\Wf \cdot \nu \cap W_\lambda \cdot \lambda = W_{\operatorname{f}, \lambda} \cdot \nu.$$The inclusion $\supset$ is straightforward. To show $\subset$, let us write $\nu = w \cdot \lambda$, for some $w \in W_\lambda$. Suppose that we have $yw \cdot \lambda = w' \cdot \lambda$, for some $y \in \Wf, w' \in W_\lambda$. Then we have $w^{-1} y^{-1} w' \cdot \lambda = \lambda$. By Lemma \ref{stabaff}, $w^{-1} y^{-1} w' \in W_\lambda$. As $w, w' \in W_\lambda$, it follows that $y \in W_\lambda$ as well. Therefore, it remains to show that \begin{equation} \label{intp} W_\lambda \cap \Wf = W_{\operatorname{f}, \lambda}.\end{equation}The inclusion $\supset$ is again clear. For the inclusion $\subset$, it suffices to show that $W_\lambda \cap \Wf$ is generated by the reflections $s_{\halpha}, \halpha \in \Phf,$ it contains. We will show that $W_\lambda \cap \Wf$ is in fact a parabolic subgroup of $W_\lambda$. 

To see this, recall that the simple positive roots in $\Phi^\vee_\lambda$ were those $\halpha \in \Phi^\vee_{\lambda, +, re}$ which could not be written as $\halpha' + \halpha''$, for $\halpha', \halpha'' \in \Phi^\vee_{\lambda, +}$, and similarly for $\Phf_\lambda$. Since ${}^{\operatorname{f}}\Phi^\vee$ is the root system of a Levi of $\Phi^\vee$, it follows that $\halpha \in {}^{\operatorname{f}}\Phi^\vee_{\lambda, +}$ is simple in $\Phf_\lambda$ if and only if it is simple in $\Phi^\vee_\lambda$. Thus, $W_{\operatorname{f}, \lambda}$ is a parabolic subgroup of $W_\lambda$. In particular, it may be characterized as those $w \in W_\lambda$ which preserve $\Phi^\vee_{\lambda, +} \setminus {}^{\operatorname{f}}\Phi^\vee_{\lambda, +}.$ But any $w \in W_\lambda \cap W_{\operatorname{f}}$ shares this property, and hence $W_\lambda \cap \Wf = W_{\operatorname{f}, \lambda}$, as desired. This completes the proof of (i). 

To prove (ii), it suffices to show that $W_{\operatorname{f}, \lambda} = W_{\operatorname{f}, \nu}$. But recalling that $W_\lambda = W_\nu$, we may use Equation \eqref{intp} to conclude $$W_{\operatorname{f}, \lambda} = W_\lambda \cap \Wf = W_\nu \cap \Wf = W_{\operatorname{f}, \nu},$$as desired.
\end{proof}

With these preparations, we may prove the linkage principle. 
\begin{theo} Let $\kappa$ be noncritical, and $\overline{\lambda} \in W_{\operatorname{f}} \backslash \h^*_\kappa$. Then: \label{glp}

\begin{enumerate}
    \item For any lift $\lambda \in \h^*_\kappa$ of $\overline{\lambda}$, the linkage class of $\overline{\lambda}$ is given by the image of $W_\lambda \cdot \lambda$ in $W_{\operatorname{f}} \backslash \h^*_\kappa$. 
    \item Write $W^\circ_\lambda$ for the stabilizer of $\lambda$ under the level $\kappa$ dot action of $W_\lambda$ on $\h^*_\kappa$. Then the highest weights in the linkage class of $\overline{\lambda}$ are parametrized by $W_{\operatorname{f}, \lambda} \backslash W_\lambda / W^\circ_\lambda$.
    
\end{enumerate}
\end{theo}

\begin{proof} We start with (1). Let $\overline{\nu} \in \Wf \backslash \h^*_\kappa$ be arbitrary, and $\overline{\lambda'}$ lie in the image of $W_\lambda \cdot \lambda$. It suffices to show that if either $\overline{\nu} \prec \overline{\lambda'}$ or $\overline{\nu} \succ \overline{\lambda'}$, cf. Section \ref{slinkies} for the notation, then $\overline{\nu}$ is also in the image of $W_\lambda \cdot \lambda$. As the statement of the theorem is unchanged by replacing $\overline{\lambda}$ with $\overline{\lambda'}$, we may take $\overline{\lambda'} =\overline{\lambda}$. 

First consider the case $\overline{\nu} \prec \overline{\lambda}$, i.e. $[M(\overline{\lambda}): L(\overline{\nu})] > 0$. Before Drinfeld-Sokolov reduction, it follows that $[M(\lambda): L(\nu)] > 0$ for some lift $\nu$ of $\overline{\nu}$. By the block decomposition for $\gk$, $\nu$ lies in $W_\lambda \cdot \lambda$, as desired. 

Next consider the case $\overline{\lambda} \prec \overline{\nu}$. Consider an arbitrary lift $\nu \in \h^*_\kappa$ of $\overline{\nu}$. Before Drinfeld-Sokolov reduction, we deduce that $[M(\nu): L(\eta)] > 0$, for some $\Wf$ antidominant $\eta$ lying in $\Wf \cdot \lambda$. If we write $\eta  = y \cdot \lambda$, $y \in \Wf$, then the block for $\gk$ passing through $\eta$ is $W_{\eta} \cdot \eta = y W_{\lambda} \cdot \lambda$. In particular, $\nu$ lies in $yW_\lambda \cdot \lambda$, and hence $\overline{\nu}$ lies in the image of $W_\lambda \cdot \lambda$. This concludes the proof of (1).  

Statement (2) follows from combining (1) and Lemma \ref{inorbits}. Namely, we may identify $W_\lambda \cdot \lambda$ with $W_\lambda/W^\circ_\lambda$, and Lemma \ref{inorbits} identifies the $\Wf$ antidominant weights in $W_\lambda \cdot \lambda$ with $W_{\operatorname{f}, \lambda} \backslash W_\lambda/ W^\circ_\lambda$. \end{proof}

\subsection{Some conjectures on the blocks of $\cW$} \label{cbl} 
Let $\lambda$ be antidominant, and consider $W_{\operatorname{f}, \lambda} \backslash W_\lambda / W^\circ_\lambda$. This exhibits the highest weights of its block as a double coset of a Coxeter group by two parabolic subgroups. We conjecture the block of $\OO$ only depends on this. To this end, it likely admits a description as appropriate monodromic parabolic-singular Soergel modules. In particular, blocks should have graded lifts and Koszul duality. 

A related expectation is that one should have translation functors relating various blocks for $\cW$, which should realize some of the above equivalences. This was mentioned as a folklore conjecture in \cite{bw3}. With some care, one can construct translation functors  from those for $\gk$ via affine Skyrabin, as we hope to explain elsewhere. We were informed by Creutzig of an alternative approach to translation functors which has been developed in forthcoming work by him and Arakawa.

Next, let $\lambda$ be regular antidominant, so that we are dealing with a one sided coset $W_{\operatorname{f}, \lambda} \backslash W_\lambda$. In this case, our conjecture asserts this block is equivalent to a monodromic singular block of $\OO$ for the appropriate Kac--Moody algebra. Moreover, at least in this case  we conjecture that Drinfeld--Sokolov reduction should identify with a translation functor. To state this precisely, suppose for simplicitly that $\lambda$ is integral, and write $\OO_{\lambda}^{mon}$ for the corresponding block of monodromic Category $\OO$ for $\gk$, and $\OO_{\cw_\kappa, \lambda}$ for the corresponding block of $\OO$ for $\cW_\kappa$. Note that $\OO^{mon}_{\lambda}$ may be identified with the free monodromic Hecke algebra, and hence carries an involution $\iota$ coming from inversion on the group. Finally, write $\mathscr{T}$ for a translation functor from $\OO^{mon}_{\lambda}$ to a $\Wf$ singular block of monodromic Category $\OO$, which we denote by $\OO_{ -\rho}^{mon}$. Then, we conjecture an equivalence $\OO_{\cw_\kappa, \lambda} \simeq \OO^{mon}_{-\rho}$ fitting into a commutative diagram $$\xymatrix{\OO^{mon}_\lambda \ar@{-}[r]^{\sim}_{\iota} \ar[d]_{\operatorname{DS}_-} & \OO^{mon}_{\lambda} \ar[d]^{\mathscr{T}} \\ \OO_{\cw_\kappa, \lambda} \ar@{-}[r]^{\sim} & \OO_{- \rho}^{mon}.}$$For nonintegral $\lambda$ one should replace $\gk$ by a Kac-Moody algebra corresponding to its integral Weyl group with a appropriately modified torus, or consider the neutral block of the corresponding twisted free monodromic Hecke algebra for $\gk$. We suspect that several but not all cases of this conjecture should follow from our forthcoming work on the tamely ramified FLE.

 Analogous statements for $\lambda$ dominant should be true. In particular, there should exist appropriately defined positive level Soergel modules. Moreover all of these should have non-monodromic variants, given a good definition for non-monodromic Category $\OO$ for $\cW$. For Virasoro, translation functors were pursued and Koszulity was studied in \cite{naka}, \cite{wies}. 

Similar stories should hold for affine $\cW$-algebras associated to other nilpotent elements principal in a Levi. Finally, we were happy to learn that for the finite $\cW$-algebras, a completely parallel story has been largely been worked out \cite{bgk}, \cite{ginz}, \cite{los}, \cite{ms},  \cite{bw}. We thank Ben Webster for bringing this to our attention.

\begin{re} For ease of notation, we wrote this section over the complex numbers $\mathbb{C}$. However, the same results hold, with similar proofs, for any algebraically closed field $k$ of characteristic zero. Namely, one may deduce the case of $\overline{\mathbb{Q}}$ from that of $\mathbb{C}$. Lemmas \ref{fstab}, \ref{stabaff} may then be proven for $k$ by choosing a basis for $k$ over $\overline{\mathbb{Q}}$, in a similar to manner to how we deduced them for $\mathbb{C}$ from the statements for $\mathbb{R}$. \end{re}

\section{Verma embeddings}
In this section, we prove de Vost--van Driel's conjecture on homomorphisms between Verma modules.
 
 \subsection{Negative level}
 \label{neglev}
 In this subsection, $\kappa$ will be negative, i.e. $\kappa \notin \kappa_c + \mathbb{Q}^{\geqslant 0} \kappa_b$. We will now determine the homomorphisms between all Verma modules for $\cw_\kappa$. 
 
 For Verma modules $M, M'$ in distinct blocks, $\Hom(M, M') = 0$, cf. Lemma \ref{nohom}. Therefore we will fix a $\overline{\lambda} \in \Wf \backslash \h^*_\kappa$, and study the Verma modules within its block. Recall that if we pick a lift $\lambda \in \h^*_\kappa$ of $\overline{\lambda}$, the highest weights of the block through $\overline{\lambda}$ are parametrized by $W_{\operatorname{f}, \lambda} \backslash W_\lambda / \wstab.$ For a double coset $w \in W_{\operatorname{f}, \lambda} \backslash W_\lambda / \wstab$, let us write $M_{w}$ for the corresponding Verma module. 
 
 Recall that $W_\lambda$ is a Coxeter group, with simple reflections $s_i, i \in I_\lambda$. In Lemma \ref{inorbits}, it was shown that $W_{\operatorname{f}, \lambda}$ is a parabolic subgroup of $W_\lambda$. By possibly replacing $\lambda$ by a $W_\lambda$ translate, we may assume $\lambda$ is antidominant. In this case, its stabilizer $\wstab$ is again a parabolic subgroup of $W_\lambda$, cf. the proof of Lemma \ref{stabaff}. 
 
 As a quotient of a Coxeter group by two parabolic subgroups, $W_{\operatorname{f}, \lambda} \backslash W_\lambda / W^\circ_\lambda$ inherits a partial order from the Bruhat order on $W_\lambda$, which we denote by $\leqslant$. Explicitly, each double coset admits a unique minimal length representative, and one restricts the Bruhat order on $W_\lambda$ along this section. With these preliminaries, we may now state the main result of this subsection. 
 
 \begin{theo} For elements $v, w \in W_{\operatorname{f}, \lambda} \backslash W_\lambda / \wstab$, $\Hom( M_{{v}}, M_{{w}})$ is at most one dimensional, and is nonzero if and only if ${v} \leqslant {w}$. \label{vemb} \end{theo}

 \subsubsection{Negative level (i): recollections for Kac--Moody} To prove Theorem \ref{vemb}, we will need the analogous story for Kac--Moody.

 \begin{pro} Fix an antidominant weight $\lambda \in \h^*_\kappa$, and for $w \in W_\lambda/\wstab$ write $M_w$ for the Verma module $M(w \cdot \lambda)$. Then:
 
 \begin{enumerate}
     \item The antidominant Verma $M_e$ is simple, and $[M_y: M_e] = 1$ for any $y \in W_\lambda / \wstab$. 
     \item $\Hom(M_y, M_w)$ is at most one dimensional, and is nonzero if and only if $y \leqslant w$ in the Bruhat order on $W_\lambda /\wstab $.
 \end{enumerate}
 \label{vembgk}
 \end{pro}

 \begin{proof} We first recall that $W_\lambda$ is the Weyl group of a sum of Kac--Moody algebras of finite and affine type. To see this, let $\mathscr{I} \subset I_\lambda$ correspond to an indecomposable summand, i.e. a connected component of the Dynkin diagram.  Write $\delta$ for the indecomposable positive imaginary root of $\gk$. Then the Cartan matrix of $\mathscr{I}$ is of infinite type if and only if $\delta$ lies in the real span of the $\alpha_\iota, \iota \in \mathscr{I}$;  moreover, if $\delta$ lies in the real span of $\alpha_i, i \in \mathscr{I}$, it in fact may be written as $\sum_{i} c_i \halpha_i$, with all $c_i$ positive, cf. Lemmas 2.2 and 2.3 of \cite{kt}.  Since $\langle \delta, \halpha_i \rangle = 0, i \in \mathscr{I}$, it follows that $\mathscr{I}$ is of affine type by Proposition 4.7 of \cite{kitty}. 
 
 By a result of Fiebig \cite{pf}, the Proposition is reduced to the analogous assertion for a negative integral block for the Kac--Moody algebra of type $I_\lambda$. Moreover, it is straightforward to reduce the assertion to the analogous one for each $\mathscr{I} \subset I_\lambda$ which is indecomposable. To do so, one may use the corresponding tensor product factorization of Verma and simple highest weight modules. 
 
 Therefore, the Proposition is reduced to the case of an integral negative block for an affine Lie algebra, or an integral block of a finite dimensional simple Lie algebra. As the latter is well documented, we only discuss the former. By using translation functors from a regular to a singular block,  we may further reduce to the case of a regular integral block.
 
 It remains to verify the Proposition for a regular integral block at negative level for an affine Lie algebra $\mathfrak{l}$ with Weyl group $W_{\mathfrak{l}}$. For (1), that $M_e$ is simple goes back to Kac--Kazhdan \cite{kk}. That $[M_y: M_e] = 1$ for $y \in W_{\mathfrak{l}}$ is a consequence of Kazhdan--Lusztig theory. Namely, $[M_y:M_e]$ is given by the inverse Kazhdan--Lusztig polynomial $Q_{e,y}(1)$. To see that $Q_{e,y}(1) = 1$, one may e.g. use the interpretation of $Q_{e,y}(1)$ at positive level and compare with the Weyl--Kac character formula. 
  
  To see (2), the analogous claim at positive level is proved in Proposition 2.5.5 of \cite{kt2}. To deduce the claim at negative level, one may apply Kac--Moody duality by the argument of Lemma \ref{duverma}. Alternatively, the characterization of when there are nonzero intertwiners follows from Kac--Kazhdan, and their uniqueness follows from (1), cf. the proof of Theorem \ref{vemb}. \end{proof}

 \subsubsection{Negative level (ii): embeddings for $\cW$}
  First, let us show:

\begin{pro} Let $\kappa$ be arbitrary. A nonzero morphism between Verma modules for $\cw_\kappa$ is injective. \label{vembb}
\end{pro}
 
 \begin{proof} Let $M(\lambda)$ be a Verma module. Recall the $\Z$-graded current algebra $\mathfrak{U}(\mathscr{W})$, cf. \cite{ara}. The usual filtration on $\mathscr{W}$ induces a filtration on the associated Lie algebra of Fourier modes and a filtration $F^i \mathfrak{U}(\mathscr{W}), i \geqslant 0$. Write $W_i, 1 \leqslant i \leqslant \operatorname{rk} \fg$, for the usual generators of $\cW_\kappa$. Write $W_i(n), n \in \mathbb{Z},$ for the Fourier mode of $W_i$ of degree $n$, and $\overline{W}_i(n)$ for its symbol in $\gr \mathfrak{U}(\mathscr{W})$. Then  $\gr \mathfrak{U}(\mathscr{W})$ contains the polynomial algebra: $$P := \C[\overline{W}_i(n)],  \quad \quad 1 \leqslant i \leqslant \operatorname{rk} \fg, \quad n > 0. $$$M(\lambda)$ is cyclically generated as a $\mathfrak{U}(\mathscr{W})$ by its highest weight vector $v_\lambda$. Under the induced filtration $F^i M(\lambda), i \geqslant 0,$ of $M(\lambda)$, $\gr M(\lambda)$ is exhibited as a free $P$ module with generator the symbol of $v_\lambda$, cf. Section 5.1 of {\em loc. cit.}. 
 
 Given a nonzero map $\phi: M(\lambda) \rightarrow M(\nu)$, write $F^n M(\nu)$ for the minimal filtered piece of $M(\nu)$ containing $v_\lambda$. It follows that $\phi$ carries $F^i M(\lambda)$ into $F^{i + n} M(\nu)$. It suffices to check injectivity on the associated graded, which is clear since this is a nonzero map of rank one $P$ modules. \end{proof}

 Having gathered the necessary ingredients, we will determine the homomorphisms between Verma modules for $\cW$ at negative level. 
 
 \begin{proof}[Proof of Theorem \ref{vemb}] We first claim that $M_{{e}}$ is simple and $[M_{{v}}: M_{{e}}] = 1$ for all ${v}$  in $\dcs$. To see this, one may apply the Drinfeld--Sokolov reduction $\operatorname{DS}_-$, cf. Theorem \ref{mrd}, to Proposition \ref{vembgk}(1). 
 
 We next claim that $\Hom( M_{{e}}, M_{{v}})$ is one dimensional. Indeed, that it is at most one dimensional follows from the facts that $[M_{{v}}: M_{{e}}] = 1$ and that $M_{{v}}$ is of finite length. That it is one dimensional follows by applying $\operatorname{DS}_-$ to a nonzero morphism in $\gk$, which exists by Proposition \ref{vembgk}(2). That it remains nonzero follows from the exactness of $\operatorname{DS}_-$. 
 
 We next claim show $\Hom(M_{{v}}, M_{{w}})$ is at most one dimensional, for any ${v}, {w} \in \dcs$. To see this, suppose one has $f,g \in \Hom(M_{{v}}, M_{{w}}).$ By the preceding paragraph, we may multiply $f$ by a scalar and assume that $f - g$ vanishes on the copy of $M_{{e}}$ embedded in $M_{{v}}$. As a nonzero homomorphism between Verma modules is an embedding, cf. Proposition \ref{vembb}, it follows that $f = g$, as desired.

 We finally claim that $\Hom(M_{{v}}, M_{{w}})$ is nonzero if and only if ${v} \leqslant {w}.$ Let us write $s(v), s(w) \in W_\lambda$ for their minimal length coset representatives. Suppose that $v \leqslant w$, i.e. that $s(v) \leqslant s(w)$. By Proposition \ref{vembgk}, there is an embedding $M_{s(v)} \rightarrow M_{s(w)}$. Applying $\operatorname{DS}_-$, one obtains an embedding $M_v \rightarrow M_w$. For the reverse implication, suppose one has an embedding $M_v \rightarrow M_w$. In particular, one has $[M_w: L_v] > 0$. Before applying $\operatorname{DS}_-$, one has $[M_{s(w)}: L_{s'(v)}] > 0$, where $s'(v)$ is the unique $\Wf$ antidominant preimage of $v$ in $W_\lambda \cdot \lambda$, cf. Lemma \ref{inorbits}. It follows that $s'(v) \leqslant s(w)$. As $s(v) \leqslant s'(v)$ by standard Coxeter combinatorics, we have $v \leqslant w$, as desired. \end{proof}
 
 \begin{re}Although it is not necessary for the proof, one can show that $$s, s': \dcs \rightarrow W_\lambda$$coincide, i.e. $s = s'$.  \end{re}

 \subsection{Positive level}
 \label{poslev}
 We now explain how to deduce the classification of Verma embeddings at positive level from that at negative level via Feigin--Fuchs duality. 
 As a preparatory observation, note that under the linear action of $\Wf$ on $\h^*$, negation sends $\Wf$ orbits into $\Wf$ orbits. Similarly, under the dot action of $\Wf$ on $\h^*$, shifted negation $\nu \rightarrow - \nu - 2 \rho$ descends to an involution of $\Wf \backslash \h^*$.

 \begin{theo}For $\kappa$ noncritical, and $\lambda \in \Wf \backslash \h^*_\kappa$,   write $M_\kappa(\lambda)$ for the Verma module for $\cW_\kappa$ of highest weight $\lambda$. Then under Feigin-Fuchs duality, we have: $$\mathbb{D} M_\kappa(\lambda) \simeq M_{-\kappa + 2\kappa_c}(- \lambda - 2 \rho)[\dim N].$$In particular, writing $\operatorname{Verma}_\kappa^\heartsuit$ for the full subcategory of $\cW_\kappa\operatorname{-mod}^\heartsuit$ consisting of Verma modules, we have a contravariant equivalence $\operatorname{Verma}_\kappa^{\heartsuit, op} \simeq \operatorname{Verma}_{-\kappa + 2 \kappa_c}^\heartsuit.$ \label{vermsd}
 \end{theo}
 
 \begin{proof}We first check that $M_\kappa(\lambda)$ is a compact object of $\cW_\kappa\operatorname{-mod}$. Write $\mathbf{I}, \mathbf{I}^-$ for the Iwahori subgroups of $L^+G$ corresponding to the preimages of $B$ and $B^-$ under the evaluation map $L^+G \rightarrow G$, respectively. Recall we wrote $I'_1$ for the prounipotent radical of $\mathbf{I}$, and $I_1 := \operatorname{Ad}_{t^{-\check{\rho}}} I'_1$. Similarly, let us write $I^{' -}_1$ for the prounipotent radical of $\mathbf{I}^-$, and $I_1^- := \operatorname{Ad}_{t^{-\check{\rho}}} I^{', -}_1$.
 
 Following the conventions of \cite{r}, we view the minus Drinfeld--Sokolov reduction $\operatorname{DS}_-$ as 
 usual reduction, up to a shift, applied to $\OO$ for $\gk$ defined with respect to $\operatorname{Ad}_{t^{-\check{\rho}}} \mathbf{I}^-$ rather than $\mathbf{I}$. I.e., we consider the composition $$\gk\operatorname{-mod}^{I'_1} \simeq \gk\operatorname{-mod}^{\operatorname{Ad}_{t^{-\check{\rho}}}I^{', -}_1} \xrightarrow{\operatorname{ins}[\dim N - \Delta]} \gk\operatorname{-mod}_{LN, \psi}.$$
 Explicitly, let us write $\mathfrak{i}^-$ for the Lie algebra of $I_1^-$. Then for $\nu \in \h^*_\kappa$, we set$$M_\kappa(\nu) = \operatorname{ind}_{\operatorname{Ad}_{t^{-\check{\rho}}}L^+ \fg}^{\gk} \operatorname{ind}_{\mathfrak{i}^-}^{\operatorname{Ad}_{t^{-\check{\rho}}}L^+\fg } \C_{w_\circ \nu}.$$Accordingly, if we write $\Lambda \in \h^*_\kappa$ for a lift of $\lambda$, then by Theorem \ref{mrd} and affine Skyrabin $M_\kappa(\lambda)$ is the image of the Verma module $M_\kappa(\Lambda)$ of $\gk$ under the composition: \begin{equation} \label{comp} \gk\operatorname{-mod}^{I_1^-} \xrightarrow{\operatorname{Av}_{I_1, \psi}} \gk\operatorname{-mod}^{I_1, \psi} \xrightarrow{\operatorname{ins}[\dim N - \Delta]} \gk\operatorname{-mod}_{LN, \psi}.\end{equation}For any category $\scc$ with a strong action of $D_\kappa(LG)$, the averaging functors between the $I_1^-$ and $(I_1,\psi)$ invariants are adjoint, up to a shift by the real dimension of $N$: \begin{equation}\label{adj}\operatorname{Av}_{I_1, \psi, *}: \scc^{I_1^-} \leftrightarrows \scc^{I_1, \psi}: \operatorname{Av}_{I_1^-, *}[2 \dim N],\end{equation}cf. the discussion in \cite[$\S$7]{r}. In particular, $\operatorname{Av}_{I_1, \psi}: \scc^{I_1^-} \rightarrow \scc^{I_1, \psi}$ preserves compactness, as it admits a continuous right adjoint. As insertion from any step of the adolescent Whittaker construction preserves compactness, we deduce that \eqref{comp} preserves compactness. Finally, since we are working with the renormalized derived category of Kac--Moody representations, $M_\kappa(\Lambda)$ is a compact object of $\gk\operatorname{-mod}$. Since $I_1^-$ is prounipotent, it is also a compact object of $\gk\operatorname{-mod}^{I_1^-}$, cf. Proposition \ref{cinu}.  
 
 Having shown that $M_\kappa(\lambda)$ is compact, we will calculate its dual by using \eqref{comp} and Proposition \ref{followduals}. Accordingly, we need to calculate the dual of $M_\kappa(\Lambda)$. 
 
 \begin{lemma}Under the perfect pairing: $$C^{\frac{\infty}{2} + *}(\hat{\mathfrak{g}}_{2\kappa_c}, L^+\mathfrak{g}, - \otimes -): \gk\operatorname{-mod} \otimes \hat{\mathfrak{g}}_{-\kappa + 2\kappa_c}\operatorname{-mod} \rightarrow \operatorname{Vect},$$the dual of $M_\kappa(\Lambda)$ is, up to isomorphism, $M_{-\kappa + 2\kappa_c}(- \Lambda - 2 \rho)[\dim N]$. 
 \label{duverma}
 \end{lemma}
Note the existing statements of the lemma in the literature, written for the usual category $\OO$, i.e. for $\mathbf{I}$, appear with an incorrect $\rho$ shift \cite{ag},\cite{gg},\cite{cc}. 
 \begin{proof} For a varying object $V$ of $\hat{\mathfrak{g}}_{-\kappa + 2\kappa_c}\operatorname{-mod}$, we need to corepresent: $$C^{\frac{\infty}{2} + *}( \hat{\mathfrak{g}}_{2\kappa_c}, L^+ \fg, M_\kappa(\Lambda) \otimes V) $$
 \begin{equation}\label{midl}\simeq C^{\frac{\infty}{2} + *}(\hat{\mathfrak{g}}_{2\kappa_c}, L^+ \fg, \operatorname{ind}_{\operatorname{Ad}_{t^{-\check{\rho}}} L^+ \fg} ^{\hat{\fg}_{2\kappa_c}} \operatorname{ind}_{\mathfrak{i}^-}^{\operatorname{Ad}_{t^{-\check{\rho}}}L^+ \fg}( \mathbb{C}_{w_\circ \Lambda} \otimes V)).\end{equation}
 
$$ \simeq C^{\frac{\infty}{2} + *}( \hat{\fg}_{2\kappa_c}, \operatorname{Ad}_{t^{-\check{\rho}}} L^+ \fg,\operatorname{ind}_{\operatorname{Ad}_{t^{-\check{\rho}}} L^+ \fg} ^{\hat{\fg}_{2\kappa_c}} \operatorname{ind}_{\mathfrak{i}^-}^{\operatorname{Ad}_{t^{-\check{\rho}}}L^+ \fg}( \mathbb{C}_{w_\circ \Lambda} \otimes V))  \otimes \operatorname{rel.det}( L^+ \fg, \operatorname{Ad}_{t^{-\check{\rho}}} L^+ \fg)$$The appearing relative determinant lies in cohomological degree zero, as follows from considering cancelling pairs of finite roots $\pm \alpha$, and we will trivialize it: \begin{equation}\cong C^{\frac{\infty}{2} + *}( \hat{\fg}_{2\kappa_c}, \operatorname{Ad}_{t^{-\check{\rho}}} L^+ \fg,\operatorname{ind}_{\operatorname{Ad}_{t^{-\check{\rho}}} L^+ \fg} ^{\hat{\fg}_{2\kappa_c}} \operatorname{ind}_{\mathfrak{i}^-}^{\operatorname{Ad}_{t^{-\check{\rho}}}L^+ \fg}( \mathbb{C}_{w_\circ \Lambda} \otimes V)).\label{med}\end{equation}
Since $\operatorname{Ad}_{t^{-\check{\rho}}} L^+ \fg$ has no nontrivial one dimensional representations, its splitting into $\hat{\fg}_{2\kappa_c}$ is unique, and therefore the canonical such.\footnote{This is not true for the Lie algebra of the Iwahori $\mathfrak{i} = \operatorname{Lie}(\mathbf{I})$, and indeed the canonical section of $\mathfrak{i}$ into $\hat{\fg}_{2 \kappa_c}$ is {\em not} given by the sequence $\mathfrak{i} \rightarrow L^+ \fg \rightarrow \hat{\fg}_{2 \kappa_c}$, but instead differs by a factor of $2\rho$ on the torus. This is why we broke the induction into two steps in \eqref{midl}, and likely where the $\rho$'s disappeared in \cite{ag}, \cite{gg}.} In particular, we may apply Proposition \ref{shap} to rewrite Equation \eqref{med} as: $$\simeq C^*( \operatorname{Ad}_{t^{-\check{\rho}}} L^+ \fg,  \operatorname{ind}_{\mathfrak{i}^-}^{\operatorname{Ad}_{t^{-\check{\rho}}}L^+ \fg} \mathbb{C}_{w_\circ \Lambda} \otimes V).$$
By the usual Shapiro lemma for Lie algebra cohomology, we may rewrite this as: $$\simeq C^*( \mathfrak{i}^-, \mathbb{C}_{w_\circ \Lambda} \otimes V \otimes \det( \operatorname{Ad}_{t^{-\check{\rho}}}L^+ \fg / \mathfrak{i}^- [1])^\vee) \simeq \operatorname{Hom}_{\mathfrak{i}^-}(\mathbb{C}_{-w_0 \Lambda} \otimes \det( \operatorname{Ad}_{t^{-\check{\rho}}}L^+ \fg / \mathfrak{i}^- [1] ), V) $$$$\simeq \operatorname{Hom}_{\mathfrak{i}^-}(\mathbb{C}_{-w_\circ \Lambda + 2 \rho} [\dim N], V)\simeq \operatorname{Hom}_{\hat{\fg}_{-\kappa + 2 \kappa_c}}(\operatorname{ind}_{\operatorname{Ad}_{t^{-\check{\rho}}} L^+ \fg}^{\hat{\fg}_{-\kappa + 2 \kappa_c}} \operatorname{ind}_{\mathfrak{i}^-}^{\operatorname{Ad}_{t^{-\check{\rho}}}L^+ \fg} \mathbb{C}_{-w_\circ \Lambda + 2 \rho}[\dim N], V),$$as desired. 
\end{proof}
 
Having calculated the dual of a Verma module for $\gk$, we may now apply Proposition \ref{followduals} and obtain a commutative diagram  \begin{equation}\label{dids}\xymatrix{ \gk\operatorname{-mod}^{I_1^-, c, op} \ar[rr]^{\operatorname{Av}_{I_1, \psi, *}} \ar[d]^{\mathbb{D}} && \gk\operatorname{-mod}^{I_1, \psi, c, op} \ar[d]^{\mathbb{D}} \\ \hat{\fg}_{-\kappa + 2\kappa_c}\operatorname{-mod}^{I_1^-, c} \ar[rr]^{((\operatorname{Av}_{I, \psi, *})^R)^\vee} && \hat{\fg}_{-\kappa + 2\kappa_c}\operatorname{-mod}^{I_1, \psi,c}.}\end{equation}  
Recall from Equation \eqref{adj} that the right adjoint $(\operatorname{Av}_{I, \psi, *})^R$ is $\operatorname{Av}_{I_1^-, *}[2\dim N]$. Recalling that $\operatorname{Av}_{I_1^-, *}$ is the composition: $$\scc^{I_1, \psi, *} \xrightarrow{\operatorname{Oblv}} \scc \xrightarrow{\operatorname{Av}_{I_1^-, *}} \scc^{I_1^-},$$it follows from Propositions \ref{dualfunc}, \ref{charpairing} that $(\operatorname{Av}_{I_1^-, *}[2 \dim N])^\vee = \operatorname{Av}_{I_1, \psi, *}[2 \dim N]$. Having unwound the bottom horizontal arrow of Equation \eqref{dids}, we insert $M_\kappa(\Lambda)$ in the top left and obtain via Lemma \ref{duverma} that: $$\mathbb{D} \operatorname{Av}_{I_1, \psi, *} M_\kappa(\Lambda) \simeq \operatorname{Av}_{I_1, \psi, *} M_{-\kappa + 2\kappa_c}( -\Lambda - 2 \rho)[3 \dim N].$$
An application of Theorem \ref{compactsinwhs} handles the subsequent insertion $\operatorname{ins}: \scc^{I_1, \psi} \rightarrow \scc_{LN, \psi}$, to give that \begin{equation}\label{uno}\mathbb{D} \operatorname{ins} \operatorname{Av}_{I_1, \psi, *} M_\kappa(\Lambda) \simeq \operatorname{ins} \operatorname{Av}_{I_1, \psi, *} M_{-\kappa + 2\kappa_c}(- \Lambda - 2 \rho) [3 \dim N - 2 \Delta].\end{equation}Since at any level $\kappa_\circ$, and weight $\nu' \in \h^*$ with image $\nu \in \Wf \backslash \h^*$ we have from Equation \eqref{comp} that\begin{equation}\label{dos}\operatorname{ins} \operatorname{Av}_{I_1, \psi, *} M_{\kappa_\circ}(\nu') \simeq M_{\kappa_\circ}(\nu)[-\dim N + \Delta],\end{equation}the result follows by combining Equations \eqref{uno}, \eqref{dos}. \end{proof}
 
\begin{ex} \label{feifu} Let $\fg = \mathfrak{sl}_2$, and let us write $k \in \mathbb{C}$ for the level $k \kappa_b$. In this situation $\cW_k$ is the Virasoro vacuum algebra $Vir_{c(k)}$, where \begin{equation}c(k) = 1 - 6 \frac{(k+1)^2}{k+2}. \label{e1}\end{equation}Then in our conventions, the Drinfeld-Sokolov minus reduction of $M_\kappa(\nu), \nu = v \rho, v \in \mathbb{C}$, is the Verma module $M(c(k),\Delta(k, v))$, where the conformal dimension $\Delta$ is given by \begin{equation}\label{e2} \Delta(k, v) = \frac{(k - v)(k + 2 - v)}{4(k+2)} - \frac{k- v}{2}.\end{equation}Theorem \ref{vermsd} applied to this case gives $$\mathbb{D} M( c(k), \Delta(k, v) ) \simeq M( c(-k - 4), \Delta(-k-4, -v - 2))[1].$$Substituting in Equations \eqref{e1}, \eqref{e2} yields the familiar form of Feigin--Fuchs duality $$\mathbb{D} M(c, \Delta) \simeq M(26 - c, 1 - \Delta)[1], \quad \quad c, \Delta \in \mathbb{C}.$$ \end{ex} 
 
 \begin{re}In the setting of Example \ref{feifu}, note that for $\fsl_2$ one has $\operatorname{Ad}_{t^{-\check{\rho}}} \mathbf{I}^- = \mathbf{I}$. In particular, one could equally well use the $+$ reduction, which sends $M_\kappa(\nu)$ to $M(c(k), \Delta'(k,v))$, where $$\Delta'(k,v) = \frac{ v(v+2)}{4(k+2)} - \frac{v}{2}.$$As is visible, the two differ by the Dynkin diagram automorphism of $\widehat{\fsl}_2$. \end{re}
 
Theorem \ref{vermsd} reduces the classification of Verma embeddings at positive level to the analogous classification at negative level, which was done in the previous subsection. We now unwind the answer the at positive level. The following propositions summarize the relevant combinatorics. 

\begin{pro} Let $\kappa$ be negative, and consider the pair of weights $\lambda \in \h^*_\kappa$ and its `flip' $-\lambda - 2 \rho \in \h^*_{-\kappa + 2\kappa_c}$. Then
\label{pm} 
\begin{enumerate}
    \item $\lambda$ is antidominant if and only if $-\lambda - 2\rho$ is dominant. 
    \item  Their integral Weyl groups are the same subgroup of $W$, i.e. $W_\lambda = W_{-\lambda - 2\rho}$.

    \item Their stabilizers in the affine Weyl group coincide, i.e. $W^\circ_{\lambda} = W^\circ_{-\lambda - 2 \rho}$.
        \item Their finite integral Weyl groups are the same subgroup of $\Wf$, i.e. $W_{\operatorname{f}, \lambda} = W_{\operatorname{f}, - \lambda - 2 \rho}$.
\end{enumerate}
\end{pro}
\begin{proof} We refer the reader to Subsection \ref{combd} for any unfamiliar notation. If we pick a lift of $\lambda \in \h^*_\kappa$ to $\Lambda \in \h^* \oplus \mathbb{C} c^* \oplus \mathbb{C} D^*,$ then a lift of $-\lambda - 2\rho$ is given by $- \Lambda - 2 \hat{\rho}$. For any coroot $\halpha \in \Phi^\vee$ we then have: $$\langle \Lambda + \rho, \halpha \rangle = - \langle (-\Lambda - 2 \hat{\rho}) + \hat{\rho}, \halpha \rangle.$$The claims of the Proposition are now straightforward, as long as one remembers for (3) that their stabilizers are generated by reflections, cf. the proof of Lemma \ref{stabaff}. \end{proof}
\begin{cor}Write $\h^*_\kappa \simeq \h^*_{-\kappa + 2\kappa_c}$ for the map $\nu \rightarrow - \nu - 2 \rho$. Taking the orbit of $\lambda$ gives an embedding $W_\lambda / W_\lambda^\circ \rightarrow \h^*_\kappa$, and similarly $W_{-\lambda - \rho} / W^\circ_{-\lambda - 2\rho} \rightarrow \h^*_{-\kappa + 2\kappa_c}$. Then under the identification of the two coset spaces by Proposition \ref{pm} (2),(3), 
the following diagram commutes $$\xymatrix{\h^*_\kappa  \ar@{-}[r]^{\sim} & \h^*_{-\kappa + 2 \kappa_c} \\ W_\lambda/W_\lambda^\circ \ar[u] \ar@{=}[r] & W_{-\lambda - 2\rho} / W^\circ_{-\lambda -2\rho} \ar[u].}$$Quotienting by $\Wf$, and applying Proposition \ref{pm} (4) to identify the arising double cosets, the following diagram commutes
 \begin{equation}\xymatrix{\Wf \backslash \h^*_\kappa  \ar@{-}[r]^{\sim} & \Wf \backslash \h^*_{-\kappa + 2 \kappa_c} \\ W_{\operatorname{f}, \lambda} \backslash W_\lambda/W_\lambda^\circ \ar[u] \ar@{=}[r] & W_{-\lambda - 2 \rho, \operatorname{f}} \backslash W_{-\lambda - 2\rho} / W^\circ_{-\lambda -2\rho}. \ar[u]}\label{flipcup}\end{equation}
\end{cor}

 We may now describe Verma embeddings at positive level $\kappa \in \kappa_c + \mathbb{Q}^{> 0} \kappa_b$ as follows. As in the discussion at the beginning of Subsection \ref{neglev}, we may immediately reduce to a single block of Category $\OO$ for $\cW_\kappa$. There is a unique highest weight in the block with a dominant lift, and dwe fix such a lift $\lambda \in \h^*_\kappa$. We may accordingly parametrize the highest weights in the block by $\dcs$, and for a double coset $w \in \dcs$, we will write $M_w$ for the corresponding Verma module. 

As discussed in Subsection \ref{neglev}, $\dcs$ is a quotient of a Coxeter group by two parabolic subgroups, and in particular carries a Bruhat order $\leqslant$. Then we have:

\begin{theo} For elements $v,w \in \dcs$, $\Hom(M_v, M_w)$ is at most one dimensional, and is nonzero if and only if $v \geqslant w$. \end{theo}
\begin{proof}This follows by combining Equation \eqref{flipcup}, Theorem \ref{vermsd}, and Theorem \ref{vemb}.\end{proof}

Finally, let us note that applying the duality $\mathbf{D}$ within Category $\OO$ for $\cW_\kappa$ yields the classification of homomorphisms between the co-Verma modules $A(\nu), \nu \in \Wf \backslash \h^*$. 

\begin{cor} Let $\kappa \neq \kappa_c$, and fix two co-Verma modules $A(\mu), A(\nu)$. Then  $\Hom( A(\mu), A(\nu))$ is at most one dimensional, and any nonzero element is a surjection. \end{cor}

\begin{cor}At a negative level $\kappa$, parametrize the co-Vermas in a block as $A_v, v \in \dcs$, where $A_v := \mathbf{D} M_v$, and $M_v$ was defined in Subsection \ref{neglev}. Then $\Hom(A_v, A_w)$ is nonzero if and only if $v \geqslant w$. 
\end{cor}

\begin{cor}At a positive level $\kappa$, parametrize the co-Vermas in a  block as $A_v, v \in \dcs,$ where $A_v := \mathbf{D} M_v$, and $M_v$ was defined in Subsection \ref{poslev}. Then $\Hom(A_v, A_w)$ is nonzero if and only if $v \leqslant w$. \end{cor}

\bibliographystyle{plain}
\bibliography{sample}

 \end{document}